\documentclass{article}
\usepackage[utf8]{inputenc}
\usepackage{graphicx}
\usepackage{float}
\usepackage{amsthm}
\usepackage{amssymb}
\usepackage{amsmath}
\usepackage[frozencache,cachedir=.]{minted}
\usepackage{url}
\usepackage{mathrsfs}
\usepackage{animate}
\usepackage{tikz}
\usepackage{newpxtext}
\usepackage{transparent}
\usepackage{placeins}
\usepackage{cleveref}
\usepackage{import}

\usepackage{array}   

\usepackage[margin=1in]{geometry}
\theoremstyle{definition}
\newtheorem{theorem}{Theorem}[section]
\newtheorem{proposition}{Proposition}[section]
\newtheorem{lemma}{Lemma}[section]
\newtheorem{question}{Question}
\newtheorem{example}{Example}

\newtheorem{remark}{Remark}
\crefname{lemma}{Lemma}{Lemmas}

\newtheorem{corollary}{Corollary}[theorem]

\newtheorem{definition}{Definition}[section]
\newcommand{\st}{\text{ such that }}
\newcommand{\cs}{\text{, }}
\newcommand{\also}{\text{ and }}
\newcolumntype{L}{>{$}l<{$}}
\newcolumntype{R}{>{$}r<{$}}
\newcolumntype{C}{>{$}c<{$}}
\newcommand{\zz}{\mathbb{Z}}
\newcommand{\rr}{\mathbb{R}}

\newcommand{\lisa}[1]{{\color{blue} \sf $\heartsuit $ Lisa: [#1]}}
\newcommand{\xingyi}[1]{{\color{orange} \sf $\star $ Xingyi: [#1]}}
\newcommand{\michael}[1]{{\color{violet} \sf $\diamond $ Michael: [#1]}}
\newcommand{\michaela}[1]{{\color{magenta} \sf $\pi $ Michaela: [#1]}}
\newcommand{\ethan}[1]{{\color{brown} \sf $\clubsuit$ Ethan: [#1]}}

\title{Optimal Constructions for DNA Self-Assembly of $k$-Regular Graphs}
\author{Lisa Baek, Ethan Bove, Michael Cho, Xingyi Zhang, Leyda Almod\'ovar, \\ Amanda Harsy, Cory Johnson, Jessica Sorrells}

\begin{document}

\maketitle

\begin{abstract}
    Within biology, it is of interest to construct DNA complexes of a certain shape. These complexes can be represented through graph theory, using edges to model strands of DNA joined at junctions, represented by vertices. Because guided construction is inefficient, design strategies for DNA self-assembly are desirable. In the flexible tile model, branched DNA molecules are referred to as tiles, each consisting of flexible unpaired cohesive ends with the ability to form bond-edges. We thus consider the minimum number of tile and bond-edge types necessary to construct a graph $G$ (i.e. a target structure) without allowing the formation of graphs of lesser order, or nonisomorphic graphs of equal order. We emphasize the concept of (un)swappable graphs, establishing lower bounds for unswappable graphs. We also introduce a method of establishing upper bounds via vertex covers. We apply both of these methods to prove new bounds on rook's graphs and Kneser graphs.
\end{abstract}

\section{Introduction}

\subsection{Motivation}
The modern advancement of fields like nanotechnology and microbiology has spurred research concerning the construction of various target structures and complexes via self-assembly. As structures continue to scale down in size, it has become increasingly difficult to build at the nano-scale. One approach to this problem utilizes DNA strands. DNA strands are made of four nucleotide bases - adenine[A], thymine[T], guanine[G] and cytosine[C]. Under the Watson-Crick model for base pairing, DNA exhibits complementary base pairing, where adenine[A] bonds with thymine[T] and guanine[G] bonds with cytosine[C].  At the ends of each DNA strand, there exist `sticky ends,' which can join with the sticky ends of another strand through complementary base pairing.  Thus, DNA has proven to be a particularly useful material for self-assembly of target structures due to these unique complementary properties \cite{Seeman_2016}.  
      
The applications of existing research are many, including various types of cancer treatment, gene therapy, targeted drug delivery, and biomolecular computing, among others (\cite{labean2007constructing, adleman, LABEAN200726, legs}). However, the synthesis of specific complexes in labs can be a costly endeavor.  As a result, it is of great interest to find more efficient ways to synthesize target complexes and understand just how efficiently it is possible to do so. Previous research has focused on construction of a variety of preprogrammed nanostructures, from branched DNA molecules \cite{kallenbach1983immobile} to DNA and RNA knots \cite{liu2016creating}. Certain mathematical models of these constructions have led to the study of some DNA nanostructures represented as discrete graphs \cite{nyu, ellis2014minimal, ellis2019tile}. Previous work has focused primarily on common graph families. Here, we work within a specific model of DNA self-assembly to determine new results applicable to graph families with certain restrictive properties that often arise in $k$-regular graphs. 

\subsection{Graph Theory Background}

Henceforth we will use the following standard definitions from graph theory in \cite{westGT}. A \textit{graph} $G$ is defined to be a set of vertices $V(G)$ and a set of edges $E(G)$, where an \textit{edge} is defined to be a set of two vertices from the vertex set. Two vertices $u,v$ are \textit{adjacent} if they are the endpoints of an edge (i.e. if $\{u, v\}\in E(G)$), and we say they are neighbors. For every edge $\{u, v\}\in E(G)$, the vertices $u$ and $v$ are called \textit{endpoints} of the edge, and the edge is said to be \textit{incident} to $u$ and $v$. A graph is said to be \textit{$k$-regular} if each vertex has $k$ incident edges. The \textit{order} of a graph is the number of vertices, which we will denote as $\#V(G)$. In this paper, we restrict all target graphs to also be simple graphs, i.e. graphs do not contain multiple edges between vertices, or edges that connect a vertex to itself. Note that the process of self-assembly may still result in a graph with multiple edges or self-loops.

\subsection{The Flexible Tile Model for DNA Self-Assembly}

In this work, we use the flexible tile model for DNA self-assembly, studied previously by Ellis-Monaghan et al. \cite{ellis2014minimal,ellis2019tile}. The central idea in using graph theory is to represent each target complex with a graph $G$. In this model, complexes are formed via the self-assembly of flexible $k$-armed junction DNA molecules, where $k$ is a variable number of branches for each molecule. For convenience, here we reproduce many of the relevant definitions found in \cite{ellis2014minimal} and \cite{ellis2019tile}.

\begin{figure}[h]
        \centering
        \includegraphics[height=5cm]{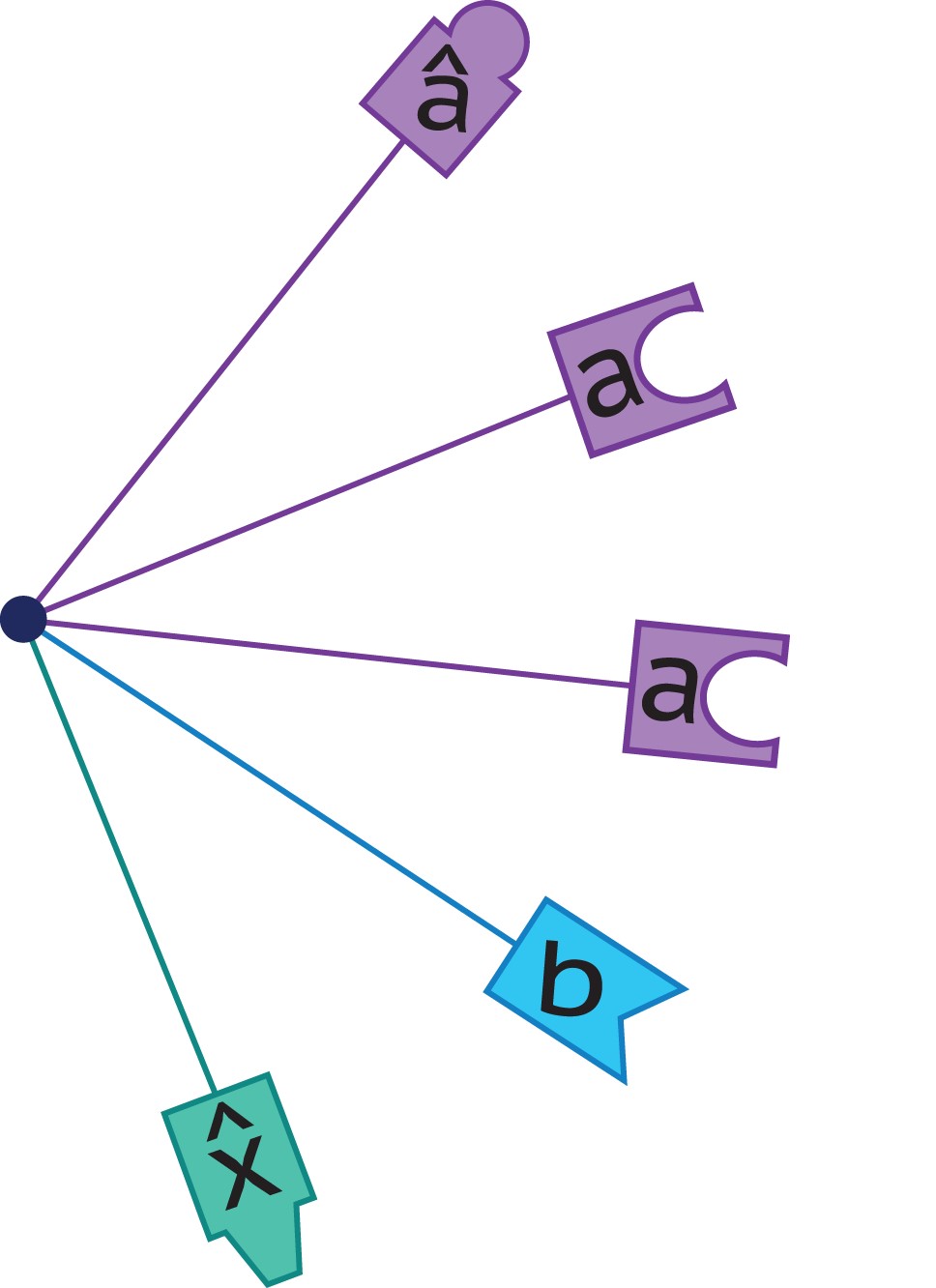}
        $\qquad$
        \includegraphics[height=5cm]{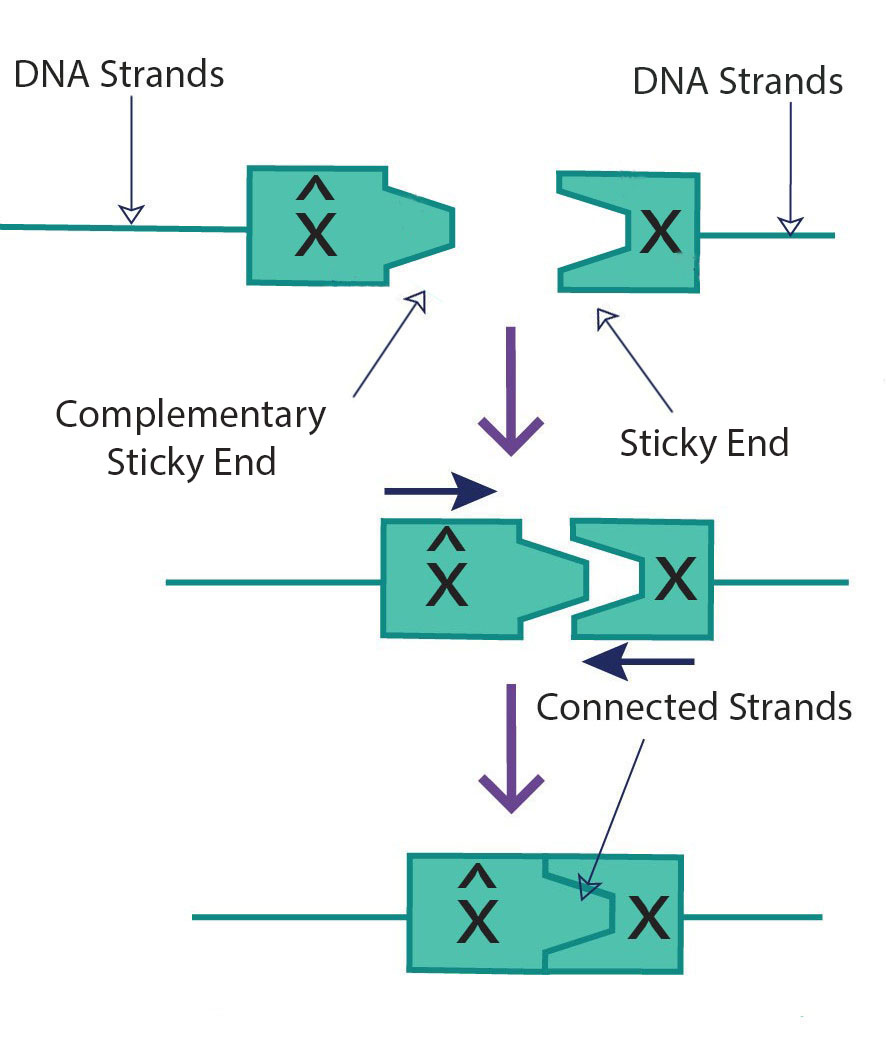}
        \label{tileexample_fig}
        \caption{Left: an example \textit{tile}, or branched junction molecule, with half-edges of bond-edge types $a$, $a$, $\hat{a}$, $b$, and $\hat{x}$. Right: an illustration of how corresponding half-edges can bond to form an edge \cite{almodovar2021complexity}.}
\end{figure}

Within the flexible tile model, \textit{tiles} represent the flexible branched molecules in the DNA structure, while each branch represents a \textit{half-edge}, often referred to as the 'sticky end' of the DNA. In particular, for each half-edge there exists a complementary half-edge to which it may bind and form a \textit{bond-edge} in the graph. For simplicity, we refer to this bond-edge and the complement with a letter and its hatted counterpart. Formally, a \emph{cohesive-end type} is an element of a finite set $S=\Sigma \cup \hat{\Sigma} $, where $\Sigma$ is the set of hatted symbols and $\hat{\Sigma}$ is the set of un-hatted symbols.  Each cohesive-end type corresponds to a distinct arrangement of bases at the end of a branched junction molecule arm. Thus, a hatted and an un-hatted symbol, say $a$ and $\hat{a}$, correspond to complementary cohesive-ends. In turn, each \textit{tile type} is given as a multiset of the cohesive-end types of its half-edges; for example, the tile in the left of Figure \ref{tileexample_fig} can be represented as $\{a^2, \hat{a}, b, \hat{x}\}$, where superscripts denote multiplicity of cohesive-end types. A cohesive-end type together with its complement is referred to as a \emph{bond-edge type}, which is typically identified by the un-hatted symbol.

Within the laboratory setting, the aim is to create the most optimal collection of branched molecule types, or tile types. A set of tile types is referred to as a \textit{pot}. In this case, ``optimal'' is used with regard to the number of distinct tile types used in the pot, as well as the number of distinct bond-edge types. The number of distinct tile types in $P$ is denoted by $\#P$, the set of bond-edge types that appear in tile types of a pot $P$ is denoted by $\Sigma(P)$, and we write $\#\Sigma(P)$ to denote the number of distinct bond-edge types that appear in $P$. These distinct bond-edge types may be illustrated in figures as different colors, which provides a connection to edge-coloring theory as seen in \cite{BF2020}. A particular pot $P$ is said to \textit{realize} a graph $G$ if we can assign a tile type in $P$ to each vertex in $G$ and its incident edges, such that every edge in $G$ is assigned a complementary pair of half-edges. In particular, note that a graph may be realized by a pot $P$ while using multiple copies of any given tile type in $P$. This is motivated by the idea that, within the laboratory setting, creation of specific tile types and bond-edge types is far more expensive than making copies of an already-existing tile type. Thus, we consider a variety of tile types, where these tile types may repeat in the self-assembled construction of a graph. For an example, consider Figure \ref{example1_fig} below; the pot $\{\{a^3\}, \{a, \hat{a}^2\}\}$ consisting of two tiles realizes the complete graph $K_4$ while using one copy of $\{a^3\}$ and three copies of $\{a, \hat{a}^2\}$.

\begin{figure}[h]
    \centering
    \includegraphics[scale=0.5]{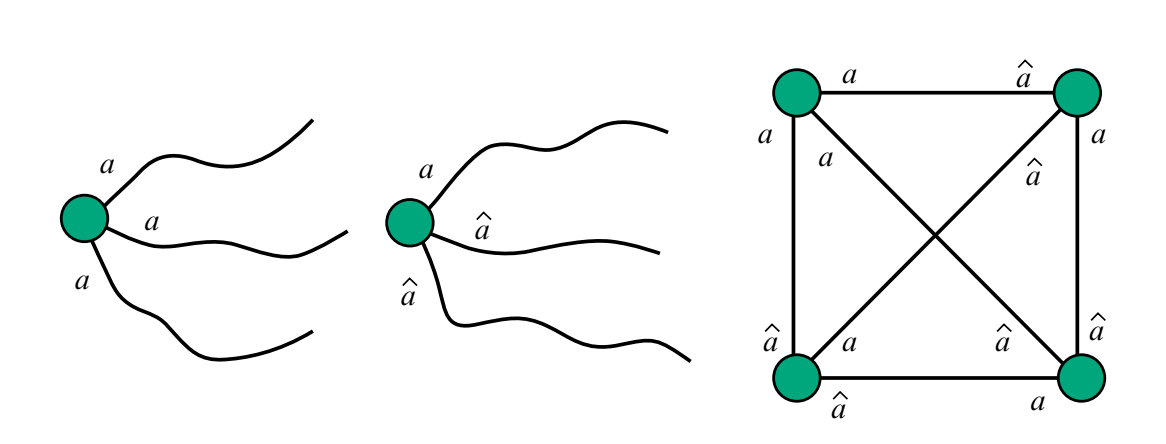}
    \caption{The pot $\{\{a^3\}, \{a, \hat{a}^2\}\}$ realizing the complete graph $K_4$ \cite{ellis2019tile}.}
    \label{example1_fig}
\end{figure}

The set of all graphs realized by a pot $P$ is denoted $\mathcal{O}(P)$, and is referred to as the \textit{output} of $P$. For a graph $G \in \mathcal{O}(P)$, when discussing a specific way that $P$ realizes $G$, we may define $\lambda : H \rightarrow \Sigma \cup \hat{\Sigma}$, where $H$ is the set of half-edges of $G$. Here $\lambda$ provides each edge with an orientation from the un-hatted half-edge to the hatted half-edge. This means that for each vertex $v \in V(G)$, $\lambda$ specifies a tile $t_v$ whose multi-set is the set of labels of half-edges incident to $v$. Thus we can use $\lambda$ to map vertices to tiles by $\lambda : V \rightarrow P_{\lambda}(G)$ such that $\lambda(v) = t_v$. For example, in Figure \ref{dumexample}, if the central vertex is denoted $v$, then $\lambda(v) = \{a^3\}$. This map $\lambda$ is referred to as the \textit{assembly design} for $G$. Furthermore, the set of all tiles used in the assembly design $\lambda$ of a graph $G$ is referred to as the \textit{assembling pot} and is denoted $P_\lambda(G)$. 

For the context of this work, we impose the following additional constraint. Given a target graph $G$, a pot $P$ that realizes $G$ must never realize another graph with fewer vertices, or a nonisomorphic graph with the same number of vertices. In other words, all graphs $G' \not\cong G$ that are realized by $P$ must be such that $\#V(G') > \#V(G)$. This constraint is often referred to as ``Scenario 3" \cite{ellis2014minimal}. Previous research, found in \cite{ellis2014minimal,ellis2019tile}, also defines Scenarios 1 and 2, where the additional constraint is significantly relaxed. However, the discussion in this work focuses solely on results for Scenario 3, and we thus omit the definitions for Scenarios 1 and 2 for brevity.

The Scenario 3 constraint is also motivated by trade-offs in the laboratory setting: self-assembly may generate smaller and different structures, and we wish to construct pots that will generate our target complex most of the time. Figure \ref{s3bad_fig} illustrates how a pot may not be valid in Scenario 3, since the pot $\{\{a^3\}, \{a, \hat{a}^2\}\}$ can realize another graph of order 4 that is not isomorphic to $K_4$.  

\begin{figure}[h]
    \label{dumexample}
    \centering
    \includegraphics[scale=0.45]{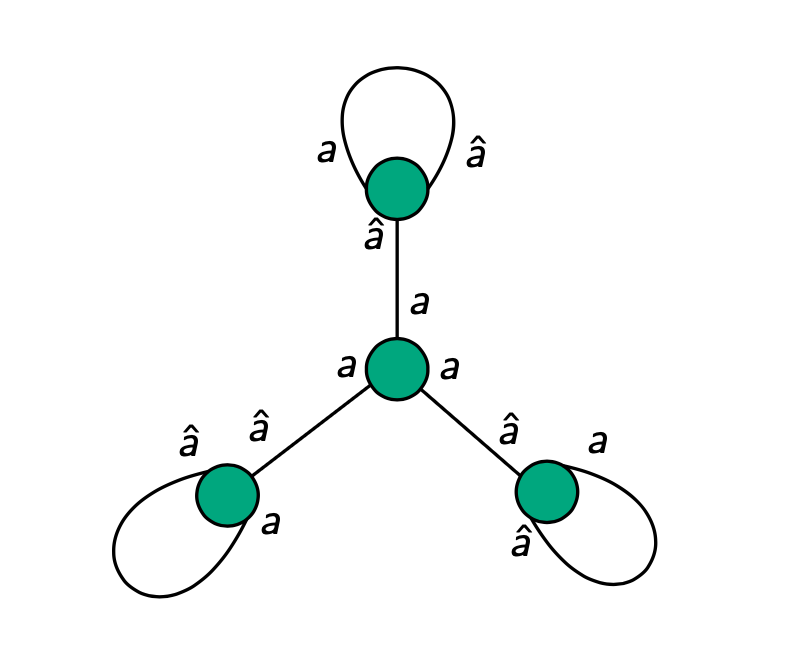}
    \caption{Another graph realized by the pot $\{\{a^3\}, \{a, \hat{a}^2\}\}$, showing that this pot is not valid in Scenario 3.}
    \label{s3bad_fig}
\end{figure}

For a graph $G$, we denote the minimum number of bond-edge types required in a pot that realizes $G$ as $B_3(G)$ and the minimum number of tile types required as $T_3(G)$. The subscript of 3 follows convention from previous research, referring to the Scenario 3 constraint. 

\subsection{Research Directions}

Some general results are known under this model \cite{BF2020, nyu}, along with some results for specific families of graphs in Scenario 3, such as cycle graphs and complete bipartite graphs \cite{ellis2014minimal}. Additional graph families have been explored in \cite{almodovar2019triangular, redmon2023optimal, griffin2023tile, almodovar2021optimal}. However, there still exist many open problems in finding optimal constructions for target complexes. In particular, the values of $B_3(G)$ and $T_3(G)$ are not known for arbitrary graphs (not to mention how to construct a valid pot for a given graph). In this paper, we are motivated by seeking a better understanding of the minimum number of bond-edge types and tile types necessary to construct complexes in the form of $k$-regular graphs for arbitrary $k$. In Section \ref{sec: lower}, we introduce and apply the concept of (un)swappable graphs to establish lower bounds on $B_3(G)$ and $T_3(G)$ for graphs with this property. It is observed that $k$-regular graphs often possess this particular property. On the other hand, Section \ref{sec: upper} provides a generalized construction method for (un)swappable graphs. While this method is not always the most efficient, it provides upper bounds on $B_3(G)$ and $T_3(G)$. Section \ref{section4} contains example applications of these methods for two specific $k$-regular graph families, including rook's graphs and Kneser graphs. We achieve exact results for sufficiently large rook's graphs and Kneser graphs of the form $Kn(n, 2)$. 

\section{Lower Bounds} \label{sec: lower}

Due to the constraints of Scenario 3, any pot that realizes a target graph must not be able to realize any nonisomorphic graph of the same (or lesser) order. However, when the tiles of the pot adjoin cohesive ends to realize a graph, there are often many different ways for half-edges to be matched with complementary half-edges. In particular, one way that half-edges could be matched differently is that two disjoint edges of the same bond-edge type (for example, two edges of bond-edge type $a$) could have their endpoints ``swapped.'' For example, two edges of bond-edge type $a$ could have their $a$ and $\hat{a}$ ends swapped, as shown in the following figure.

\begin{figure}[H]
    \centering
    \includegraphics[scale=0.65]{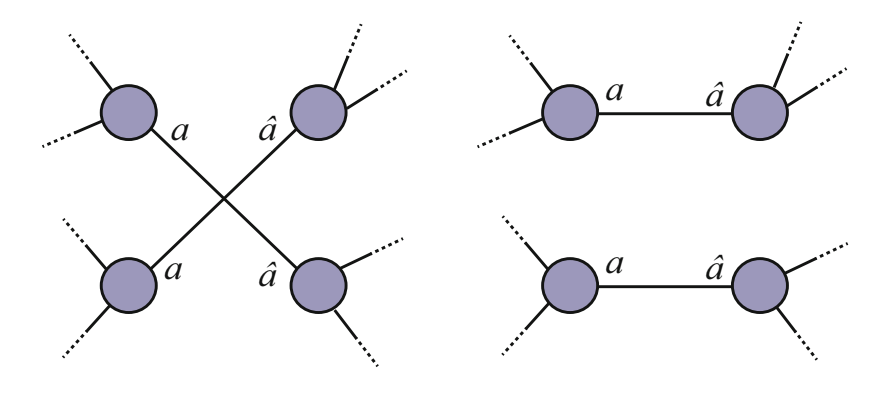}
    \caption{Half-edges may be swapped during self-assembly \cite{ellis2014minimal}.}
\end{figure}

Swapping these edges in this way creates a graph that can also be realized by the same pot, since the half-edges of every vertex is matched with the correct corresponding half-edge in the new graph as well (each vertex stays as the same tile type). This property was introduced in \cite{ellis2014minimal} for Scenario 3 via the following lemma: 

\begin{lemma}[Ellis-Monaghan et al.]
\label{swappinglemma}
Suppose $P$ is a pot that realizes graph $G$ in Scenario 3. If two disjoint edges $\{u, v\}$ and $\{s, t\}$ of $G$ use the same bond-edge type, then $G$ must be isomorphic to $G' = G - \{\{u, v\},\{s, t\}\} + \{\{u, t\}, \{s, v\}\}$.
\end{lemma}

Since the newly formed graph must also be realized by the pot, the new graph must be isomorphic to the target graph (as long as the pot fulfills Scenario 3). However, in a given graph $G$, it is not necessarily the case that two edges can be swapped to create an isomorphic graph. Motivated by the above, we define the following property of graphs.

\begin{definition}
We say that a graph $G$ is \textbf{unswappable} if, for all pairs of disjoint edges $\{u, v\}$ and $\{s, t\}$, $G$ is not isomorphic to $G' = G - \{\{u, v\},\{s, t\}\} + \{\{u, t\}, \{s, v\}\}$. We only refer to graphs as unswappable if they have at least 4 vertices, since smaller graphs cannot have disjoint edges to swap. Otherwise, the graph is called \textbf{swappable}.
\end{definition}

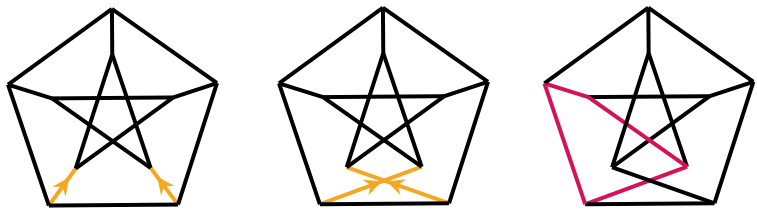
\begin{figure}[H]
        \centering
        \tikzset{every picture/.style={line width=0.75pt}} 

\begin{tikzpicture}[x=0.75pt,y=0.75pt,yscale=-1,xscale=1]

\draw [line width=1.5]    (93.99,13.03) -- (94.17,36) ;
\draw [line width=1.5]    (147,50.44) -- (124.93,57.7) ;
\draw [color={rgb, 255:red, 245; green, 166; blue, 35 }  ,draw opacity=1 ][line width=1.5]    (75.72,93.57) -- (62.19,112.25) ;
\draw [shift={(72.41,98.13)}, rotate = 125.91] [fill={rgb, 255:red, 245; green, 166; blue, 35 }  ,fill opacity=1 ][line width=0.08]  [draw opacity=0] (8.75,-4.2) -- (0,0) -- (8.75,4.2) -- (5.81,0) -- cycle    ;
\draw [line width=1.5]    (41.58,51.24) -- (63.76,58.17) ;
\draw [color={rgb, 255:red, 245; green, 166; blue, 35 }  ,draw opacity=1 ][line width=1.5]    (113.52,93.29) -- (127.34,111.76) ;
\draw [shift={(116.9,97.8)}, rotate = 53.21] [fill={rgb, 255:red, 245; green, 166; blue, 35 }  ,fill opacity=1 ][line width=0.08]  [draw opacity=0] (8.75,-4.2) -- (0,0) -- (8.75,4.2) -- (5.81,0) -- cycle    ;
\draw [line width=1.5]    (41.58,51.24) -- (93.99,13.03) ;
\draw [line width=1.5]    (75.72,93.57) -- (124.93,57.7) ;
\draw [line width=1.5]    (63.76,58.17) -- (124.93,57.7) ;
\draw [line width=1.5]    (94.17,36) -- (113.52,93.29) ;
\draw [line width=1.5]    (62.19,112.25) -- (127.34,111.76) ;
\draw [line width=1.5]    (93.99,13.03) -- (147,50.44) ;
\draw [line width=1.5]    (147,50.44) -- (127.34,111.76) ;
\draw [line width=1.5]    (41.58,51.24) -- (62.19,112.25) ;
\draw [line width=1.5]    (63.76,58.17) -- (113.52,93.29) ;
\draw [line width=1.5]    (75.72,93.57) -- (94.17,36) ;
\draw [line width=1.5]    (230.67,12.39) -- (230.85,35.35) ;
\draw [line width=1.5]    (283.67,49.79) -- (261.61,57.06) ;
\draw [color={rgb, 255:red, 245; green, 166; blue, 35 }  ,draw opacity=1 ][line width=1.5]    (250.2,92.64) -- (198.87,111.61) ;
\draw [shift={(230.07,100.08)}, rotate = 159.72] [fill={rgb, 255:red, 245; green, 166; blue, 35 }  ,fill opacity=1 ][line width=0.08]  [draw opacity=0] (8.75,-4.2) -- (0,0) -- (8.75,4.2) -- (5.81,0) -- cycle    ;
\draw [line width=1.5]    (178.26,50.59) -- (200.44,57.52) ;
\draw [color={rgb, 255:red, 245; green, 166; blue, 35 }  ,draw opacity=1 ][line width=1.5]    (212.39,92.93) -- (264.02,111.11) ;
\draw [shift={(232.64,100.06)}, rotate = 19.41] [fill={rgb, 255:red, 245; green, 166; blue, 35 }  ,fill opacity=1 ][line width=0.08]  [draw opacity=0] (8.75,-4.2) -- (0,0) -- (8.75,4.2) -- (5.81,0) -- cycle    ;
\draw [line width=1.5]    (178.26,50.59) -- (230.67,12.39) ;
\draw [line width=1.5]    (212.39,92.93) -- (261.61,57.06) ;
\draw [line width=1.5]    (200.44,57.52) -- (261.61,57.06) ;
\draw [line width=1.5]    (230.85,35.35) -- (250.2,92.64) ;
\draw [line width=1.5]    (198.87,111.61) -- (264.02,111.11) ;
\draw [line width=1.5]    (230.67,12.39) -- (283.67,49.79) ;
\draw [line width=1.5]    (283.67,49.79) -- (264.02,111.11) ;
\draw [line width=1.5]    (178.26,50.59) -- (198.87,111.61) ;
\draw [line width=1.5]    (200.44,57.52) -- (250.2,92.64) ;
\draw [line width=1.5]    (212.39,92.93) -- (230.85,35.35) ;
\draw [line width=1.5]    (364.52,12.39) -- (364.69,35.35) ;
\draw [line width=1.5]    (417.52,49.79) -- (395.45,57.06) ;
\draw [color={rgb, 255:red, 212; green, 17; blue, 89 }  ,draw opacity=1 ][line width=1.5]    (384.04,92.64) -- (332.71,111.61) ;
\draw [color={rgb, 255:red, 212; green, 17; blue, 89 }  ,draw opacity=1 ][line width=1.5]    (312.1,50.59) -- (334.28,57.52) ;
\draw [color={rgb, 255:red, 0; green, 0; blue, 0 }  ,draw opacity=1 ][line width=1.5]    (346.24,92.93) -- (397.86,111.11) ;
\draw [line width=1.5]    (312.1,50.59) -- (364.52,12.39) ;
\draw [line width=1.5]    (346.24,92.93) -- (395.45,57.06) ;
\draw [line width=1.5]    (334.28,57.52) -- (395.45,57.06) ;
\draw [line width=1.5]    (364.69,35.35) -- (384.04,92.64) ;
\draw [line width=1.5]    (332.71,111.61) -- (397.86,111.11) ;
\draw [line width=1.5]    (364.52,12.39) -- (417.52,49.79) ;
\draw [line width=1.5]    (417.52,49.79) -- (397.86,111.11) ;
\draw [color={rgb, 255:red, 212; green, 17; blue, 89 }  ,draw opacity=1 ][line width=1.5]    (312.1,50.59) -- (332.71,111.61) ;
\draw [color={rgb, 255:red, 212; green, 17; blue, 89 }  ,draw opacity=1 ][line width=1.5]    (334.28,57.52) -- (384.04,92.64) ;
\draw [line width=1.5]    (346.24,92.93) -- (364.69,35.35) ;

\end{tikzpicture}
        \caption{Illustration that the Petersen graph is unswappable.}
        \label{petersen_unswap}
    \end{figure}

\begin{figure}[H]
        \centering
        \tikzset{every picture/.style={line width=0.75pt}} 

\begin{tikzpicture}[x=0.75pt,y=0.75pt,yscale=-1,xscale=1]

\draw [color={rgb, 255:red, 245; green, 166; blue, 35 }  ,draw opacity=1 ][line width=1.5]    (40.2,100.6) -- (40.2,30.8) ;
\draw [shift={(40.2,61.3)}, rotate = 90] [fill={rgb, 255:red, 245; green, 166; blue, 35 }  ,fill opacity=1 ][line width=0.08]  [draw opacity=0] (8.75,-4.2) -- (0,0) -- (8.75,4.2) -- (5.81,0) -- cycle    ;
\draw [color={rgb, 255:red, 245; green, 166; blue, 35 }  ,draw opacity=1 ][line width=1.5]    (110,100.6) -- (110,30.8) ;
\draw [shift={(110,61.3)}, rotate = 90] [fill={rgb, 255:red, 245; green, 166; blue, 35 }  ,fill opacity=1 ][line width=0.08]  [draw opacity=0] (8.75,-4.2) -- (0,0) -- (8.75,4.2) -- (5.81,0) -- cycle    ;
\draw [color={rgb, 255:red, 245; green, 166; blue, 35 }  ,draw opacity=1 ][line width=1.5]    (160.2,100.6) -- (230,30.8) ;
\draw [shift={(195.1,65.7)}, rotate = 135] [fill={rgb, 255:red, 245; green, 166; blue, 35 }  ,fill opacity=1 ][line width=0.08]  [draw opacity=0] (8.75,-4.2) -- (0,0) -- (8.75,4.2) -- (5.81,0) -- cycle    ;
\draw [color={rgb, 255:red, 245; green, 166; blue, 35 }  ,draw opacity=1 ][line width=1.5]    (230,100.6) -- (160.2,30.8) ;
\draw [shift={(195.1,65.7)}, rotate = 45] [fill={rgb, 255:red, 245; green, 166; blue, 35 }  ,fill opacity=1 ][line width=0.08]  [draw opacity=0] (8.75,-4.2) -- (0,0) -- (8.75,4.2) -- (5.81,0) -- cycle    ;
\draw [line width=1.5]    (110,100.6) -- (40.2,100.6) ;
\draw [line width=1.5]    (230,100.6) -- (160.2,100.6) ;
\draw [line width=1.5]    (230,30.8) -- (160.2,30.8) ;
\draw [line width=1.5]    (110,30.8) -- (40.2,30.8) ;

\draw (127,53.8) node [anchor=north west][inner sep=0.75pt]    {$\simeq $};
\end{tikzpicture}
        \caption{Illustration that the cycle graph $C_4$ is swappable.}
        \label{4cycleswap}
    \end{figure}
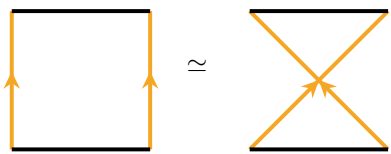

Examples of unswappable and swappable graphs are shown in Figures \ref{petersen_unswap} and \ref{4cycleswap}, respectively. In the case of the Petersen graph, the smallest cycle has length 5, but swapping any two disjoint edges creates a cycle of length 4 \cite{https://doi.org/10.1112/blms.12154}, making this graph unswappable. In the case of the 4-cycle, swapping opposing edges results in an isomorphic graph, making this graph swappable.

Graphs with the unswappable property can thus never fulfill the hypothesis of Lemma \ref{swappinglemma}, since they cannot be isomorphic to any graph produced by a swap. Therefore, if $P$ is a pot that realizes an unswappable graph $G$ in Scenario 3, then two disjoint edges cannot use the same bond-edge type. Note that this is essentially equivalent to using the contrapositive of the above lemma. We may extend this argument to the following:

\begin{proposition} \label{prop: allsinks}
    If a pot $P$ realizes an unswappable graph $G$ in Scenario 3, then either all half-edges of a given cohesive-end type $a \in \Sigma(P)$ or all half-edges of its complement $\hat{a}$ will appear in a single tile type. 
\end{proposition}
\begin{proof}
Let $P$ realize an unswappable graph $G$ with assembly design $P_\lambda(G)$. Suppose, for contradiction, that there exist two non-adjacent vertices $v_1,v_2$ such that $a \in \lambda(v_1)$ and $a \in \lambda(v_2)$ (that is, these vertices both have half-edges labeled with the same cohesive-end type). Then there exist two non-incident edges $e_1$ and $e_2$ in $G$ (where $v_1$ and $v_2$ are endpoints of these edges, respectively) labeled with bond-edge type $a$. If these bond-edges are ``broken" and  the corresponding half-edges re-join in the alternate manner, we have ``swapped" two edges in $G$ and a nonisomorphic graph $G'$ is formed. The tile types of $P$ remain unchanged, and thus $G' \in \mathcal{O}(P)$ and $P$ is invalid in Scenario 3.
\end{proof}

Note that replacing all $a$'s with $\hat{a}$'s and vice versa in a pot does not change the fundamental structure of the pot. Thus, without loss of generality, we may assume by Proposition \ref{prop: allsinks} that for each bond-edge type $a$ in a pot realizing an unswappable graph, all half-edges of cohesive-end type $a$ are on a single tile type. In particular, the argument used in the proof of Proposition \ref{prop: allsinks} shows that this tile corresponds to a single vertex in the graph in order to prevent the existence of two disjoint edges labeled with bond-edge type $a$. For a given bond-edge type $a$, we refer to this vertex (incident to all edges labeled with cohesive-end type $a$) as the \emph{source} of $a$. This also implies that the tile type with all half-edges of cohesive-end type $a$ can only be used once in an assembly design of an unswappable graph. A labeling of a graph illustrating these constraints is shown in Figure \ref{puppy}.

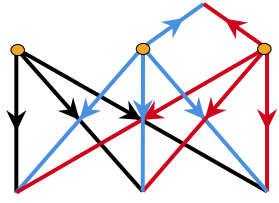
\begin{figure}[H]
        \centering
        \begin{tikzpicture}[x=0.75pt,y=0.75pt,yscale=-.8,xscale=.8]

\draw [line width=1.5]    (99.4,69.77) -- (99.4,160.6) ;
\draw [shift={(99.4,121.98)}, rotate = 270] [fill={rgb, 255:red, 0; green, 0; blue, 0 }  ][line width=0.08]  [draw opacity=0] (13.4,-6.43) -- (0,0) -- (13.4,6.44) -- (8.9,0) -- cycle    ;
\draw [color={rgb, 255:red, 208; green, 2; blue, 27 }  ,draw opacity=1 ][line width=1.5]    (257.45,69.06) -- (257.45,159.9) ;
\draw [shift={(257.45,121.28)}, rotate = 270] [fill={rgb, 255:red, 208; green, 2; blue, 27 }  ,fill opacity=1 ][line width=0.08]  [draw opacity=0] (13.4,-6.43) -- (0,0) -- (13.4,6.44) -- (8.9,0) -- cycle    ;
\draw [line width=1.5]    (99.4,69.77) -- (179.08,159.9) ;
\draw [shift={(139.24,114.83)}, rotate = 228.52] [fill={rgb, 255:red, 0; green, 0; blue, 0 }  ][line width=0.08]  [draw opacity=0] (13.4,-6.43) -- (0,0) -- (13.4,6.44) -- (8.9,0) -- cycle    ;
\draw [line width=1.5]    (99.4,69.77) -- (257.45,159.9) ;
\draw [shift={(178.42,114.83)}, rotate = 209.69] [fill={rgb, 255:red, 0; green, 0; blue, 0 }  ][line width=0.08]  [draw opacity=0] (13.4,-6.43) -- (0,0) -- (13.4,6.44) -- (8.9,0) -- cycle    ;
\draw [color={rgb, 255:red, 74; green, 144; blue, 226 }  ,draw opacity=1 ][line width=1.5]    (179.08,69.06) -- (99.4,160.6) ;
\draw [shift={(139.24,114.83)}, rotate = 311.04] [fill={rgb, 255:red, 74; green, 144; blue, 226 }  ,fill opacity=1 ][line width=0.08]  [draw opacity=0] (13.4,-6.43) -- (0,0) -- (13.4,6.44) -- (8.9,0) -- cycle    ;
\draw [color={rgb, 255:red, 208; green, 2; blue, 27 }  ,draw opacity=1 ][line width=1.5]    (257.45,69.06) -- (99.4,160.6) ;
\draw [shift={(178.42,114.83)}, rotate = 329.92] [fill={rgb, 255:red, 208; green, 2; blue, 27 }  ,fill opacity=1 ][line width=0.08]  [draw opacity=0] (13.4,-6.43) -- (0,0) -- (13.4,6.44) -- (8.9,0) -- cycle    ;
\draw [color={rgb, 255:red, 74; green, 144; blue, 226 }  ,draw opacity=1 ][line width=1.5]    (257.45,159.9) -- (179.08,69.06) ;
\draw [shift={(218.26,114.48)}, rotate = 229.21] [fill={rgb, 255:red, 74; green, 144; blue, 226 }  ,fill opacity=1 ][line width=0.08]  [draw opacity=0] (13.4,-6.43) -- (0,0) -- (13.4,6.44) -- (8.9,0) -- cycle    ;
\draw [color={rgb, 255:red, 208; green, 2; blue, 27 }  ,draw opacity=1 ][line width=1.5]    (257.45,69.06) -- (179.08,159.9) ;
\draw [shift={(218.26,114.48)}, rotate = 310.79] [fill={rgb, 255:red, 208; green, 2; blue, 27 }  ,fill opacity=1 ][line width=0.08]  [draw opacity=0] (13.4,-6.43) -- (0,0) -- (13.4,6.44) -- (8.9,0) -- cycle    ;
\draw [color={rgb, 255:red, 74; green, 144; blue, 226 }  ,draw opacity=1 ][line width=1.5]    (217.5,41.6) -- (179.08,69.06) ;
\draw [shift={(205.04,50.5)}, rotate = 144.44] [fill={rgb, 255:red, 74; green, 144; blue, 226 }  ,fill opacity=1 ][line width=0.08]  [draw opacity=0] (13.4,-6.43) -- (0,0) -- (13.4,6.44) -- (8.9,0) -- cycle    ;
\draw [color={rgb, 255:red, 208; green, 2; blue, 27 }  ,draw opacity=1 ][line width=1.5]    (217.5,41.6) -- (257.45,69.77) ;
\draw [shift={(230.69,50.9)}, rotate = 35.18] [fill={rgb, 255:red, 208; green, 2; blue, 27 }  ,fill opacity=1 ][line width=0.08]  [draw opacity=0] (13.4,-6.43) -- (0,0) -- (13.4,6.44) -- (8.9,0) -- cycle    ;
\draw  [fill={rgb, 255:red, 245; green, 166; blue, 35 }  ,fill opacity=1 ] (96,70.75) .. controls (96,68.81) and (97.9,67.23) .. (100.25,67.23) .. controls (102.6,67.23) and (104.5,68.81) .. (104.5,70.75) .. controls (104.5,72.7) and (102.6,74.27) .. (100.25,74.27) .. controls (97.9,74.27) and (96,72.7) .. (96,70.75) -- cycle ;
\draw  [fill={rgb, 255:red, 245; green, 166; blue, 35 }  ,fill opacity=1 ] (251.5,70.05) .. controls (251.5,68.1) and (253.4,66.53) .. (255.75,66.53) .. controls (258.1,66.53) and (260,68.1) .. (260,70.05) .. controls (260,71.99) and (258.1,73.57) .. (255.75,73.57) .. controls (253.4,73.57) and (251.5,71.99) .. (251.5,70.05) -- cycle ;
\draw [color={rgb, 255:red, 74; green, 144; blue, 226 }  ,draw opacity=1 ][line width=1.5]    (179.27,70.75) -- (179.08,159.9) ;
\draw [shift={(179.18,115.32)}, rotate = 270.13] [fill={rgb, 255:red, 74; green, 144; blue, 226 }  ,fill opacity=1 ][line width=0.08]  [draw opacity=0] (13.4,-6.43) -- (0,0) -- (13.4,6.44) -- (8.9,0) -- cycle    ;
\draw  [fill={rgb, 255:red, 245; green, 166; blue, 35 }  ,fill opacity=1 ] (175.03,70.05) .. controls (175.03,68.1) and (176.93,66.53) .. (179.27,66.53) .. controls (181.62,66.53) and (183.52,68.1) .. (183.52,70.05) .. controls (183.52,71.99) and (181.62,73.57) .. (179.27,73.57) .. controls (176.93,73.57) and (175.03,71.99) .. (175.03,70.05) -- cycle ;

\end{tikzpicture}
        \caption{Graph labeling in accordance with Proposition \ref{prop: allsinks}. Source vertices are identified in orange.}
        \label{puppy}
    \end{figure}

When an unswappable graph $G$ is realized by a pot, we will denote the set of all source vertices by $S$. Note that each edge $e \in E(G)$ must be labeled with some bond-edge type $a$, and therefore $e$ is incident to $a$'s source vertex. This motivates the introduction of the following definition, commonly used in graph theory:

\begin{definition}
    A \textbf{vertex cover} of a graph $G$ is a set of vertices $K\subseteq V(G)$ such that every edge is incident to at least one vertex in $K$. An example may be found in Figure \ref{petersen_vc}. 
\end{definition}

\begin{figure}[H]
        \centering
        \tikzset{every picture/.style={line width=0.75pt}} 

\begin{tikzpicture}[x=0.75pt,y=0.75pt,yscale=-1,xscale=1]

\draw [line width=1.5]    (135.73,16.79) -- (136,52.36) ;
\draw [line width=1.5]    (216.78,74.73) -- (183.03,85.98) ;
\draw [line width=1.5]    (107.78,141.55) -- (87.09,170.48) ;
\draw [line width=1.5]    (55.58,75.97) -- (89.49,86.7) ;
\draw [line width=1.5]    (165.59,141.1) -- (186.72,169.72) ;
\draw [line width=1.5]    (55.58,75.97) -- (135.73,16.79) ;
\draw [line width=1.5]    (107.78,141.55) -- (183.03,85.98) ;
\draw [line width=1.5]    (89.49,86.7) -- (183.03,85.98) ;
\draw [line width=1.5]    (136,52.36) -- (165.59,141.1) ;
\draw [line width=1.5]    (87.09,170.48) -- (186.72,169.72) ;
\draw [line width=1.5]    (135.73,16.79) -- (216.78,74.73) ;
\draw [line width=1.5]    (216.78,74.73) -- (186.72,169.72) ;
\draw [line width=1.5]    (55.58,75.97) -- (87.09,170.48) ;
\draw [line width=1.5]    (89.49,86.7) -- (165.59,141.1) ;
\draw [line width=1.5]    (107.78,141.55) -- (136,52.36) ;
\draw  [fill={rgb, 255:red, 155; green, 155; blue, 155 }  ,fill opacity=1 ][line width=1.5]  (83.3,86.03) .. controls (83.3,82.61) and (86.07,79.84) .. (89.49,79.84) .. controls (92.91,79.84) and (95.69,82.61) .. (95.69,86.03) .. controls (95.69,89.45) and (92.91,92.23) .. (89.49,92.23) .. controls (86.07,92.23) and (83.3,89.45) .. (83.3,86.03) -- cycle ;
\draw  [fill={rgb, 255:red, 155; green, 155; blue, 155 }  ,fill opacity=1 ][line width=1.5]  (170.65,85.98) .. controls (170.65,82.56) and (173.42,79.79) .. (176.84,79.79) .. controls (180.26,79.79) and (183.03,82.56) .. (183.03,85.98) .. controls (183.03,89.4) and (180.26,92.17) .. (176.84,92.17) .. controls (173.42,92.17) and (170.65,89.4) .. (170.65,85.98) -- cycle ;
\draw  [fill={rgb, 255:red, 155; green, 155; blue, 155 }  ,fill opacity=1 ][line width=1.5]  (80.9,169.81) .. controls (80.9,166.39) and (83.67,163.62) .. (87.09,163.62) .. controls (90.51,163.62) and (93.29,166.39) .. (93.29,169.81) .. controls (93.29,173.23) and (90.51,176.01) .. (87.09,176.01) .. controls (83.67,176.01) and (80.9,173.23) .. (80.9,169.81) -- cycle ;
\draw  [fill={rgb, 255:red, 155; green, 155; blue, 155 }  ,fill opacity=1 ][line width=1.5]  (180.53,169.72) .. controls (180.53,166.3) and (183.3,163.52) .. (186.72,163.52) .. controls (190.14,163.52) and (192.91,166.3) .. (192.91,169.72) .. controls (192.91,173.14) and (190.14,175.91) .. (186.72,175.91) .. controls (183.3,175.91) and (180.53,173.14) .. (180.53,169.72) -- cycle ;
\draw  [fill={rgb, 255:red, 155; green, 155; blue, 155 }  ,fill opacity=1 ][line width=1.5]  (129.81,52.36) .. controls (129.81,48.94) and (132.58,46.17) .. (136,46.17) .. controls (139.42,46.17) and (142.19,48.94) .. (142.19,52.36) .. controls (142.19,55.78) and (139.42,58.55) .. (136,58.55) .. controls (132.58,58.55) and (129.81,55.78) .. (129.81,52.36) -- cycle ;
\draw  [fill={rgb, 255:red, 155; green, 155; blue, 155 }  ,fill opacity=1 ][line width=1.5]  (129.54,16.79) .. controls (129.54,13.37) and (132.31,10.6) .. (135.73,10.6) .. controls (139.15,10.6) and (141.92,13.37) .. (141.92,16.79) .. controls (141.92,20.21) and (139.15,22.99) .. (135.73,22.99) .. controls (132.31,22.99) and (129.54,20.21) .. (129.54,16.79) -- cycle ;
\end{tikzpicture}
        \caption{A minimum vertex cover on the Petersen graph.}
        \label{petersen_vc}
    \end{figure}
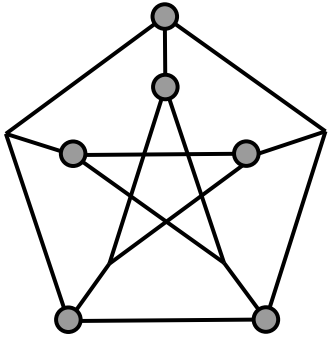

We can use the previous logic to reach the following corollary of Proposition \ref{prop: allsinks}: 
\begin{corollary}
The set $S$ of all source vertices of a labeled graph $G$ must be a vertex cover of $G$.
\end{corollary}

This means that all assembly designs corresponding to pots that realize unswappable graphs have an underlying structure of a vertex cover, which is a very helpful constraint on the construction of such pots. This allows us to extend our earlier argument to obtain the following lower bound:

\begin{theorem}\label{b3lowerbound}
Let $G$ be unswappable and $K$ be a minimum vertex cover of $G$. Then, $B_3(G) \geq |K|$.
\end{theorem}
\begin{proof}
Let $G$ be an unswappable graph, and let $P$ be such that $P$ realizes $G$ in Scenario 3 and $\#\Sigma(P) = B_3(G)$. By Proposition \ref{prop: allsinks}, for any bond-edge type $a \in \Sigma(P)$, all edges with cohesive-end type $a$ must be incident to a single vertex in the labeling of $G$ that results from an assembly design using $P$. 

Let $B_3(G) = b$ and denote the unique bond-edge types in $\Sigma(P)$ as $a_1, ..., a_b$. Each bond-edge type $a_i$ must correspond to a vertex, call it $v_i$, which is incident with all half-edges labeled with $a_i$ (i.e. $v_i$ is the source vertex for $a_i$). Note that all edges in $G$ are incident to one of $v_1, ..., v_b$, since if an edge is labeled with bond-edge type $a_i$, then it is incident to $v_i$. Thus, $v_1, ..., v_b$ is a vertex cover for $G$ of size $b = B_3(G)$, and $B_3(G)$ is greater than or equal to the size of the minimum vertex cover of $G$.
\end{proof}

Morally, this means that an unswappable graph cannot be formed via self-assembly without using at least $|K|$ bond-edge types, where $K$ is a minimum vertex cover of the graph. For many well-known graphs and families of graphs, the size of the minimum vertex cover is a known quantity, allowing us to achieve concrete bounds through Theorem \ref{b3lowerbound}. In general, however, finding the size of the minimum vertex cover of a graph is \textsf{NP}-hard. In fact, this problem is still \textsf{NP}-hard even if the graph is known to be $k$-regular \cite{regularvc}.

Given an unswappable graph, it is natural to ask the question of whether there exists a pot that achieves this lower bound for bond-edge types. We can reverse-engineer a pot from a vertex cover by assigning vertices in the vertex cover as sources for distinct bond-edge types.

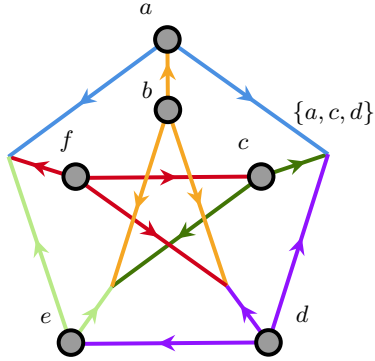
\begin{figure}[H]
               \centering
        \tikzset{every picture/.style={line width=0.75pt}} 

\begin{tikzpicture}[x=0.75pt,y=0.75pt,yscale=-1,xscale=1]

\draw [color={rgb, 255:red, 245; green, 166; blue, 35 }  ,draw opacity=1 ][line width=1.5]    (130.4,39.46) -- (130.67,75.03) ;
\draw [shift={(130.49,51.35)}, rotate = 89.56] [fill={rgb, 255:red, 245; green, 166; blue, 35 }  ,fill opacity=1 ][line width=0.08]  [draw opacity=0] (8.75,-4.2) -- (0,0) -- (8.75,4.2) -- (5.81,0) -- cycle    ;
\draw [color={rgb, 255:red, 65; green, 117; blue, 5 }  ,draw opacity=1 ][line width=1.5]    (211.44,97.4) -- (177.7,108.65) ;
\draw [shift={(200.17,101.16)}, rotate = 161.56] [fill={rgb, 255:red, 65; green, 117; blue, 5 }  ,fill opacity=1 ][line width=0.08]  [draw opacity=0] (8.75,-4.2) -- (0,0) -- (8.75,4.2) -- (5.81,0) -- cycle    ;
\draw [color={rgb, 255:red, 184; green, 233; blue, 134 }  ,draw opacity=1 ][line width=1.5]    (102.45,164.22) -- (81.76,193.15) ;
\draw [shift={(95.53,173.88)}, rotate = 125.56] [fill={rgb, 255:red, 184; green, 233; blue, 134 }  ,fill opacity=1 ][line width=0.08]  [draw opacity=0] (8.75,-4.2) -- (0,0) -- (8.75,4.2) -- (5.81,0) -- cycle    ;
\draw [color={rgb, 255:red, 208; green, 2; blue, 27 }  ,draw opacity=1 ][line width=1.5]    (50.25,98.64) -- (84.16,109.37) ;
\draw [shift={(61.58,102.23)}, rotate = 17.56] [fill={rgb, 255:red, 208; green, 2; blue, 27 }  ,fill opacity=1 ][line width=0.08]  [draw opacity=0] (8.75,-4.2) -- (0,0) -- (8.75,4.2) -- (5.81,0) -- cycle    ;
\draw [color={rgb, 255:red, 144; green, 19; blue, 254 }  ,draw opacity=1 ][line width=1.5]    (160.26,163.77) -- (181.38,192.39) ;
\draw [shift={(167.32,173.33)}, rotate = 53.56] [fill={rgb, 255:red, 144; green, 19; blue, 254 }  ,fill opacity=1 ][line width=0.08]  [draw opacity=0] (8.75,-4.2) -- (0,0) -- (8.75,4.2) -- (5.81,0) -- cycle    ;
\draw [color={rgb, 255:red, 74; green, 144; blue, 226 }  ,draw opacity=1 ][line width=1.5]    (50.25,98.64) -- (130.4,39.46) ;
\draw [shift={(85.58,72.56)}, rotate = 323.56] [fill={rgb, 255:red, 74; green, 144; blue, 226 }  ,fill opacity=1 ][line width=0.08]  [draw opacity=0] (8.75,-4.2) -- (0,0) -- (8.75,4.2) -- (5.81,0) -- cycle    ;
\draw [color={rgb, 255:red, 65; green, 117; blue, 5 }  ,draw opacity=1 ][line width=1.5]    (102.45,164.22) -- (177.7,108.65) ;
\draw [shift={(135.33,139.94)}, rotate = 323.56] [fill={rgb, 255:red, 65; green, 117; blue, 5 }  ,fill opacity=1 ][line width=0.08]  [draw opacity=0] (8.75,-4.2) -- (0,0) -- (8.75,4.2) -- (5.81,0) -- cycle    ;
\draw [color={rgb, 255:red, 208; green, 2; blue, 27 }  ,draw opacity=1 ][line width=1.5]    (84.16,109.37) -- (177.7,108.65) ;
\draw [shift={(135.33,108.98)}, rotate = 179.56] [fill={rgb, 255:red, 208; green, 2; blue, 27 }  ,fill opacity=1 ][line width=0.08]  [draw opacity=0] (8.75,-4.2) -- (0,0) -- (8.75,4.2) -- (5.81,0) -- cycle    ;
\draw [color={rgb, 255:red, 245; green, 166; blue, 35 }  ,draw opacity=1 ][line width=1.5]    (130.67,75.03) -- (160.26,163.77) ;
\draw [shift={(146.85,123.58)}, rotate = 251.56] [fill={rgb, 255:red, 245; green, 166; blue, 35 }  ,fill opacity=1 ][line width=0.08]  [draw opacity=0] (8.75,-4.2) -- (0,0) -- (8.75,4.2) -- (5.81,0) -- cycle    ;
\draw [color={rgb, 255:red, 144; green, 19; blue, 254 }  ,draw opacity=1 ][line width=1.5]    (81.76,193.15) -- (181.38,192.39) ;
\draw [shift={(125.67,192.81)}, rotate = 359.56] [fill={rgb, 255:red, 144; green, 19; blue, 254 }  ,fill opacity=1 ][line width=0.08]  [draw opacity=0] (8.75,-4.2) -- (0,0) -- (8.75,4.2) -- (5.81,0) -- cycle    ;
\draw [color={rgb, 255:red, 74; green, 144; blue, 226 }  ,draw opacity=1 ][line width=1.5]    (130.4,39.46) -- (211.44,97.4) ;
\draw [shift={(174.5,70.99)}, rotate = 215.56] [fill={rgb, 255:red, 74; green, 144; blue, 226 }  ,fill opacity=1 ][line width=0.08]  [draw opacity=0] (8.75,-4.2) -- (0,0) -- (8.75,4.2) -- (5.81,0) -- cycle    ;
\draw [color={rgb, 255:red, 144; green, 19; blue, 254 }  ,draw opacity=1 ][line width=1.5]    (211.44,97.4) -- (181.38,192.39) ;
\draw [shift={(198.19,139.27)}, rotate = 107.56] [fill={rgb, 255:red, 144; green, 19; blue, 254 }  ,fill opacity=1 ][line width=0.08]  [draw opacity=0] (8.75,-4.2) -- (0,0) -- (8.75,4.2) -- (5.81,0) -- cycle    ;
\draw [color={rgb, 255:red, 184; green, 233; blue, 134 }  ,draw opacity=1 ][line width=1.5]    (50.25,98.64) -- (81.76,193.15) ;
\draw [shift={(64.14,140.3)}, rotate = 71.56] [fill={rgb, 255:red, 184; green, 233; blue, 134 }  ,fill opacity=1 ][line width=0.08]  [draw opacity=0] (8.75,-4.2) -- (0,0) -- (8.75,4.2) -- (5.81,0) -- cycle    ;
\draw [color={rgb, 255:red, 208; green, 2; blue, 27 }  ,draw opacity=1 ][line width=1.5]    (84.16,109.37) -- (160.26,163.77) ;
\draw [shift={(125.79,139.13)}, rotate = 215.56] [fill={rgb, 255:red, 208; green, 2; blue, 27 }  ,fill opacity=1 ][line width=0.08]  [draw opacity=0] (8.75,-4.2) -- (0,0) -- (8.75,4.2) -- (5.81,0) -- cycle    ;
\draw [color={rgb, 255:red, 245; green, 166; blue, 35 }  ,draw opacity=1 ][line width=1.5]    (102.45,164.22) -- (130.67,75.03) ;
\draw [shift={(114.78,125.25)}, rotate = 287.56] [fill={rgb, 255:red, 245; green, 166; blue, 35 }  ,fill opacity=1 ][line width=0.08]  [draw opacity=0] (8.75,-4.2) -- (0,0) -- (8.75,4.2) -- (5.81,0) -- cycle    ;
\draw  [fill={rgb, 255:red, 155; green, 155; blue, 155 }  ,fill opacity=1 ][line width=1.5]  (77.97,108.7) .. controls (77.97,105.28) and (80.74,102.51) .. (84.16,102.51) .. controls (87.58,102.51) and (90.35,105.28) .. (90.35,108.7) .. controls (90.35,112.12) and (87.58,114.9) .. (84.16,114.9) .. controls (80.74,114.9) and (77.97,112.12) .. (77.97,108.7) -- cycle ;
\draw  [fill={rgb, 255:red, 155; green, 155; blue, 155 }  ,fill opacity=1 ][line width=1.5]  (171.51,108.65) .. controls (171.51,105.23) and (174.28,102.46) .. (177.7,102.46) .. controls (181.12,102.46) and (183.89,105.23) .. (183.89,108.65) .. controls (183.89,112.07) and (181.12,114.85) .. (177.7,114.85) .. controls (174.28,114.85) and (171.51,112.07) .. (171.51,108.65) -- cycle ;
\draw  [fill={rgb, 255:red, 155; green, 155; blue, 155 }  ,fill opacity=1 ][line width=1.5]  (75.57,192.48) .. controls (75.57,189.06) and (78.34,186.29) .. (81.76,186.29) .. controls (85.18,186.29) and (87.95,189.06) .. (87.95,192.48) .. controls (87.95,195.9) and (85.18,198.68) .. (81.76,198.68) .. controls (78.34,198.68) and (75.57,195.9) .. (75.57,192.48) -- cycle ;
\draw  [fill={rgb, 255:red, 155; green, 155; blue, 155 }  ,fill opacity=1 ][line width=1.5]  (175.19,192.39) .. controls (175.19,188.97) and (177.96,186.19) .. (181.38,186.19) .. controls (184.8,186.19) and (187.58,188.97) .. (187.58,192.39) .. controls (187.58,195.81) and (184.8,198.58) .. (181.38,198.58) .. controls (177.96,198.58) and (175.19,195.81) .. (175.19,192.39) -- cycle ;
\draw  [fill={rgb, 255:red, 155; green, 155; blue, 155 }  ,fill opacity=1 ][line width=1.5]  (124.48,75.03) .. controls (124.48,71.61) and (127.25,68.84) .. (130.67,68.84) .. controls (134.09,68.84) and (136.86,71.61) .. (136.86,75.03) .. controls (136.86,78.45) and (134.09,81.22) .. (130.67,81.22) .. controls (127.25,81.22) and (124.48,78.45) .. (124.48,75.03) -- cycle ;
\draw  [fill={rgb, 255:red, 155; green, 155; blue, 155 }  ,fill opacity=1 ][line width=1.5]  (124.2,39.46) .. controls (124.2,36.04) and (126.98,33.27) .. (130.4,33.27) .. controls (133.82,33.27) and (136.59,36.04) .. (136.59,39.46) .. controls (136.59,42.88) and (133.82,45.66) .. (130.4,45.66) .. controls (126.98,45.66) and (124.2,42.88) .. (124.2,39.46) -- cycle ;

\draw (114.8,19.74) node [anchor=north west][inner sep=0.75pt]  [font=\small] [align=left] {$\displaystyle a$};
\draw (116,59.34) node [anchor=north west][inner sep=0.75pt]  [font=\small] [align=left] {$\displaystyle b$};
\draw (164.4,88.14) node [anchor=north west][inner sep=0.75pt]  [font=\small] [align=left] {$\displaystyle c$};
\draw (193.6,172.14) node [anchor=north west][inner sep=0.75pt]  [font=\small] [align=left] {$\displaystyle d$};
\draw (64.8,176.14) node [anchor=north west][inner sep=0.75pt]  [font=\small] [align=left] {$\displaystyle e$};
\draw (74.4,84.14) node [anchor=north west][inner sep=0.75pt]  [font=\small] [align=left] {$\displaystyle f$};
\draw (192,69.34) node [anchor=north west][inner sep=0.75pt]  [font=\small]  {$\{a,c,d\}$};

\end{tikzpicture}
        \caption{An assembly design of the Petersen graph derived from the vertex cover in Figure \ref{petersen_vc}.} \label{petersen_pot}
    \end{figure}

\begin{example}
    Figure \ref{petersen_pot} provides an example of an assembly design and corresponding labeling reverse-engineered from the vertex cover given in Figure \ref{petersen_vc}, using the vertex cover as source vertices for distinct bond-edge types. Colors represent distinct bond-edge types, and edge orientation indicates which half-edges are $a$ and $\hat{a}$ (with arrows pointing away from unhatted cohesive-end types). The pot, written out explicitly, is as follows: 
       $$ P= \{\{a^2, \hat{b}\}, \{b^3\}, \{c^2, \hat{f}\}, \{d^3\}, \{e^2, \hat{d}\}, \{f^3\}, \{\hat{a}, \hat{e}, \hat{f}\}, \{\hat{a}, \hat{c}, \hat{d}\}, \{\hat{b}, \hat{c}, \hat{e}\}, \{\hat{b}, \hat{d}, \hat{f}\}\}$$
       
\end{example}

If $K$ is a vertex cover for $G$ and $\lambda$ is an assembly design derived from $K$, then note that every vertex $v\not\in K$ is only adjacent to vertices in $K$. Thus, any vertex $v\not \in K$ must be labeled such that the incident half-edges are of cohesive-end types determined by the neighboring source vertices; this determines the tile type $t$ such that $\lambda(v)=t$. Recall that a neighbor set can be defined as follows:

\begin{definition}
Given a graph $G$ and vertex $v\in G$, let $N(v)$ denote the \textbf{neighbor set} of $v$, i.e. the set of all vertices connected to $v$ via an edge. 
\end{definition}

Therefore, we find that for $v\not \in K$, the neighbor set $N(v)$ determines the tile type used to label $v$.

\begin{proposition}\label{t3lowerbound}
If $G$ is an unswappable graph with $X$ distinct neighbor sets among all vertices (i.e. $X = |\{N(v): v\in V(G)\}|$), then $T_3(G) \geq X$.
\end{proposition}

\begin{proof}
Suppose $\lambda$ is an assembly design for $G$ in Scenario 3. We know that $\Sigma(P_\lambda(G))$ must correspond to a vertex cover of $G$; in particular, for every bond-edge type, there is a designated source vertex. Then for any two source vertices $v,w$ we have $\lambda(v) \neq \lambda(w)$, since these are the unique vertices with half-edges labeled with the unhatted cohesive-end type of their designated bond-edge type.

Furthermore, since the set of source vertices is a vertex cover, we can deduce the tile type of a non-source vertex by its neighbor set (all neighbors must be in the vertex cover). In particular, if a vertex $v$ has the bond-edge type $a$ source vertex as a neighbor, then that $\hat{a} \in \lambda(v)$. Therefore, non-source vertices can only have the same tile type if they have the same neighbor set. Thus, $\#P_\lambda(G)$ is equal to the sum of the number of source vertices and the number of distinct neighbor sets among the non-source vertices. This is bounded below by the number of distinct neighbor sets among all vertices in $G$. In particular, the number of source vertices is bounded below by the number of distinct neighbor sets among the source vertices, so the sum is bounded below by the number of distinct neighbor sets among all vertices.
\end{proof}

In the case where no two vertices have the same neighbor set, we obtain the following result:

\begin{corollary}
\label{tileUpperBoundCor}
If $G$ is unswappable and for every two distinct vertices $v, w \in G$ we have that $N(v) \neq N(w)$, then $T_3(G)=\#V(G)$.
\end{corollary}

For example, the Petersen graph shown in Figure \ref{petersen_pot} requires a unique tile type for each vertex. It is worth noting that the property of unswappability is not very common among graphs in general. However, many graphs and graph families that are $k$-regular have some underlying structure that is broken when any edge swap is made -- this often arises in cases of combinatorial graphs. Section \ref{section4} gives examples of specific graph families for which these lower bounds given in Theorem \ref{b3lowerbound} and Proposition \ref{t3lowerbound} apply. Often in the case of combinatorial graphs and other unswappable graphs, they also have the property that no two vertices have the same neighbor set, giving exact values for $T_3$ as well.

One hindrance to the methods presented here is that it is difficult to check whether a given graph is unswappable, since there is no known polynomial time algorithm solving the graph isomorphism problem \cite{Babai2019GROUPGA}. The difficulty of checking whether a graph remains isomorphic after an edge swap has been demonstrated within the flexible tile model for specific graph families \cite{almodovar2019triangular} \cite{griffin2023tile}.  The authors provide a shorthand method of checking unswappability in Appendix \ref{appendixA}. Note that the method in Appendix \ref{appendixA} is not guaranteed to be correct, since we use the number of spanning trees as a proxy to determine whether graphs are isomorphic (and thus it cannot distinguish nonisomorphic graphs that have the same number of spanning trees). In practice, however, this method is a very effective sanity check for small cases.

\section{Upper Bounds} \label{sec: upper}

We would now like to construct upper bounds on graphs meeting certain conditions or, said another way, establish methods for Scenario 3 pot construction for these graphs. As we saw when establishing the lower bounds, the ``swappability" problem for a graph $G$ can be rather difficult to solve. To use the same bond-edge type to label non-incident edges requires checking isomorphism between two graphs, which is both computationally difficult and hard to generalize for large families. However, if we opt to only repeat bond-edge types on edges incident to each other, this challenge is removed.

Suppose we have a pot $P$ that realizes a target graph $G$. If we assume that all edges of a certain bond-edge type are incident to each other, then (as in the previous section) each bond-edge type will have a unique source vertex, and the set of all source vertices $K$ will form a vertex cover for $G$. This motivates the following question:

\begin{question}
    Given a graph $G$, what vertex covers $K$ can be used as source vertices in a valid Scenario 3 pot?
\end{question}

The following example shows that there are vertex covers that cannot be used as source vertices in a valid Scenario 3 pot.

\begin{example}
    Consider the cube graph $Q_3$, represented with the vertices $V(Q_3) = \{0, 1\}^3$ and edges between vertices differing at one coordinate. Then the set $K = \{(0, 0, 0), (0, 1, 1), (1, 0, 1), (1, 1, 0)\}$ is a vertex cover (in fact, a minimum one). We may assign the vertices in $K$ to be source vertices for the bond-edge types $a, b, c,$ and $d$; this assembly design $\lambda$ defines the assembling pot $$P_\lambda(Q_3) = \{\{a^3\}, \{b^3\}, \{c^3\}, \{d^3\}, \{\hat a, \hat b, \hat c\}, \{\hat a, \hat b, \hat d\}, \{\hat a, \hat c, \hat d\}, \{\hat b,\hat c, \hat d\}\}.$$ This pot is not valid in Scenario 3, however, as the subset of tiles $\{\{a^3\}, \{b^3\}, \{c^3\}, \{\hat a, \hat b, \hat c\}\}$ realizes $K_{3, 3}$. This is shown in Figure \ref{Q3ConstructionAndExample}.

    \begin{figure} 
    \centering
    \import{./}{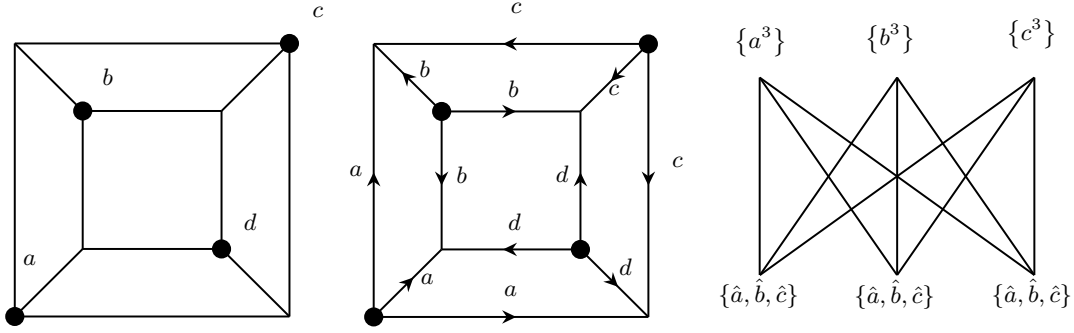}
    \caption{Vertex cover of $Q_3$ (left); labeling of $Q_3$ produced from vertex cover (center); $K_{3,3}$ graph realized by resulting pot (right).}
    \label{Q3ConstructionAndExample}
\end{figure}
\end{example}

It is important to note that, given a vertex cover $K \subseteq V(G)$, there is typically more than one pot $P$ using $K$ as source vertices. While the bond-edges between $K$ and $V(G) - K$ are determined, one must choose a bond-edge type for each edge between vertices in $K$. If such an edge is $e = \{v, w\}$, then $e$ must use either the bond-edge type associated with $v$ or the bond-edge type associated with $w$. This can be thought of as designating a source vertex for $e$. We now borrow the following two definitions from \cite{westGT}:

\begin{definition}
    Given a graph $G$ and a vertex subset $H \subseteq V(G)$, the \textbf{induced subgraph} on $H$ is the graph obtained by taking vertex set $H$ and all edges in $G$ that are between vertices of $H$. 
\end{definition}

Note that the induced subgraph and the set $H$ are often referred to interchangeably. 

\begin{definition}
    Given a graph $G$, an \textbf{orientation} $O$ on $G$ is a choice of a source and sink for every edge in $G$. 
\end{definition}
With these definitions in mind, it is clear that the additional information required to construct a pot $P$ from a vertex cover $K$ is an orientation $O$ on the induced subgraph $K$. The source vertex for an edge determines the bond-edge type used to label that edge. For clarity, we may define the following:

\begin{definition}
Given a tuple $(K, O)$ of a vertex cover of a graph $G$ and an orientation on its induced subgraph, we may construct a pot $P$ by considering each vertex in $K$ as a source vertex for unique bond-edge types and orienting edges between source vertices via the orientation $O$. In particular, the pot $P$ consists of the set of tile types corresponding to the labeled vertices of $G$, once we assign bond-edge types and orientations for each edge. This pot $P$ is said to be \textbf{derived} from the tuple $(K, O)$.
\end{definition}

Note that this ``derivation'' of a pot is equivalent to the ``reverse-engineering'' discussed in Section 2, with the additional information of the orientation $O$. The following example illustrates this definition:

\begin{example}
    Consider the Petersen Graph $Kn(5, 2)$ defined on the 2-element subsets of $\{1, \cdots, 5\}$. Take the vertex cover $K = \{\{1, 2\}, \{3, 4\}, \{1, 3\}, \{2, 4\}, \{1, 4\}, \{2, 3\}\},$ and the orientation $O$ defined as $\{1, 2\} \to \{3, 4\}, \{1, 3\} \to \{2, 4\}, \{1, 4\} \to \{2, 3\}$. We can derive the pot $P$ from $(K, O)$ as follows. First, associate to each element of the vertex cover a bond-edge type. We assign 
    $$\{1, 2\} \leftrightarrow a, \qquad \{3, 4\} \leftrightarrow b, \qquad \{1, 3\} \leftrightarrow c, \qquad 
    \{2, 4\} \leftrightarrow d, \qquad \{1, 4\} \leftrightarrow e, \qquad \{2, 3\} \leftrightarrow f.$$
    Now, for every edge not between vertices of $K$, we use the bond-edge type associated with the vertex that is in $K$. For edges within $K$, we use the bond-edge type of the vertex that is the source of $e$ in the orientation $O$. Following this assembly design, call it $\lambda$, gives the following assembling pot $P_\lambda(Kn(5,2))$:
    $$
    \lambda(\{1,2\}) = \{a^3\}, \quad \lambda(\{3,4\})=\{b^2, \hat a\}, \quad \lambda(\{1,3\}) = \{c^3\}, \quad \lambda(\{2,4\}) = \{d^2, \hat c\},
    \quad \lambda(\{1,4\}) = \{e^3\},$$ $$
    \lambda(\{2,3\})=\{f^2, \hat e\}, \quad \lambda(\{1,5\}) = \{\hat b, \hat d, \hat f\}, \quad  \lambda(\{2,5\}) = \{\hat b, \hat c, \hat e\}, \quad
    \lambda(\{3,5\}) = \{\hat a, \hat d, \hat e\}, \quad 
    \lambda(\{4,5\}) = \{\hat a, \hat c, \hat f\}.
    $$
    This construction process is illustrated in Figure \ref{PetersenVtexCoverConst}.
\end{example}

\begin{figure} 
    \centering
    \import{./}{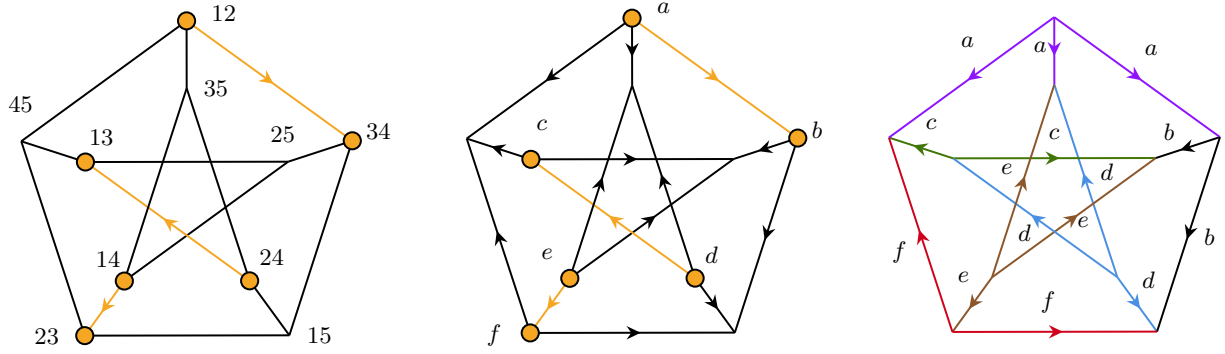}
    \caption{Vertex cover of Petersen graph with chosen edge orientations (in orange) (left); complete chosen orientation of the Petersen graph (center); resulting labeling of the Petersen graph (right). In the final construction, bond edges $a, b, c, d, e,$ and $f$ are color coded purple, black, green, blue, brown, and red, respectively. The label for each edge is also indicated by the letter closest to the edge's central arrow.}
    \label{PetersenVtexCoverConst}
\end{figure}

Given a graph $G$, we wish to know which tuples $(K, O)$ will allow us to derive a valid Scenario 3 pot for $G$. In particular, if $P$ is the pot derived from $(K, O)$, we wish to determine whether $P$ generates any nonisomorphic graphs of equal or lesser order. We have the following lemma:

\begin{lemma}
\label{directed edge lemma}
    Let $P_\lambda(G)$ be an assembling pot derived from an oriented vertex cover $(K, O)$ on a graph $G$. Suppose there is an edge $e = (v_1 , v_2)$ in $O$ such that $\lambda(v_1)=t_1$ and $\lambda(v_2)=t_2$. Then for any graph $G' \in \mathcal{O}(P_\lambda(G))$, if $\lambda'$ is an assembly design for $G'$ and $\lambda'(v) = t_2$ for some $v \in V(G')$, then there exists $w \in V(G')$ such that $\lambda'(w)=t_1$.
\end{lemma}

\begin{proof}
    Let $v_1,v_2 \in V(G)$ be sources for the bond-edges types $b_1$ and $b_2$, respectively. Since edge $e$ is directed $v_1 \to v_2$ in $O$, we have that $\lambda((v_1,v_2))=b_1$. Thus, $\hat b_1 \in t_2$. Suppose $G' \in \mathcal{O}(P_\lambda(G))$ and let $\lambda'$ be an assembly design for $G'$ such that $\lambda'(v)=t_2$ for some $v \in G'$. Since $\hat{b_1} \in t_2$, there must exist $w \in G'$ such that $b_1 \in \lambda'(w)$. As $t_1$ contains all unhatted $b_1$ cohesive-end types, it follows that $t_1 \in P_\lambda'(G')$.
\end{proof}

We can extend this argument to include not just single edges, but entire directed paths in a given orientation. 

\begin{corollary}
    \label{directed path cor}
    Let $P_\lambda(G)$ be an assembling pot derived from an oriented vertex cover $(K, O)$ on a graph $G$. Suppose $v_1, v_2 \in K$ are such that there exists a directed path $v_1 \to w_1 \to w_2 \to \cdots \to w_k \to v_2$ in $O$ and $\lambda(v_1)=t_1, \lambda(v_2)=t_2$. Then for any graph $G' \in \mathcal{O}(P_\lambda(G))$, if $\lambda'$ is an assembly design for $G'$ and $\lambda'(v)=t_2$ for some $v\in G'$, then there exists $w \in V(G')$ such that $\lambda'(w)=t_1$.
\end{corollary}

\begin{proof}
    By Lemma \ref{directed edge lemma}, if $t_2 \in P_\lambda'(G')$, then $\lambda(w_k) \in P_\lambda'(G')$ as well, since there exists a directed edge from $w_k$ to $v_2$. 
    Applying Lemma \ref{directed edge lemma} inductively through the $w_i$ vertices, we see that $\lambda(w_1) \in P_\lambda'(G')$, and thus $t_1 \in P_\lambda'(G')$. 
\end{proof}

With this result in mind, it is natural to consider orientations $O$ in which such a path exists between all pairs of vertices in $K$. In this case, an assembling pot of a graph $G'$ realized by a pot $P$ derived from a tuple $(K,O)$ would either contain all tiles in $P$ or none of them. Because any labeling of $G'$ must use some unhatted version of a cohesive-end type, and thus any resulting pot contains a tile associated with a vertex in $K$, this property would imply that all tiles associated with vertices of $K$ are used. The concept of strong connectivity describes this desirable property.

\begin{definition}
    We say that a directed graph $G$ is \textbf{strongly connected} if for every two vertices $v_1, v_2 \in V(G)$, there exists a directed path from $v_1$ to $v_2$. Similarly, an orientation of a graph $G$ is a \textit{strong orientation} if the orientation results in a strongly connected graph.
\end{definition}

To state our final result more succinctly, it's worthwhile to establish exactly which graphs $G$ can be given strong orientations. 
We recall the following definition and well-known theorem \cite{westGT}:

\begin{definition}
    We say that a graph $G$ is \textbf{2-edge-connected} if, for every edge $e \in G$, the graph $G - e$ is connected.
\end{definition}

\begin{theorem}
    Given a graph $G$, a strong orientation exists on $G$ if and only if $G$ is 2-edge-connected. 
\end{theorem}

The final hurdle for establishing our upper bound for $B_3$ comes in demonstrating that a derived pot $P$ does not realize graphs of lesser order than $G$. To resolve this problem, we invoke the following novel definition:

\begin{definition}
\label{neighborhood indepedence}
    Let $G$ be a graph and $K \subseteq V(G)$. We consider the vector space $\rr^{K}$ represented as maps $K \to \rr$. For each vertex $v \in V(G) - K$, we construct the vector $\rho_v$ by taking, for every $w \in K$:
    $$\rho_v(w) = \begin{cases}
        1 & \text{if } \{v, w\} \in E(G) \\
        0 & \text{else}
    \end{cases}.$$
    We say that $K$ is \textbf{neighborhood independent} if the set of vectors $\{\rho_v\}_{v \in V(G) - K}$ is linearly independent.
\end{definition}

It's important to note that this definition requires only for the set of vectors $\{\rho_v\}_{v \in V(G) - K}$, in the formal definition of a set, to be linearly independent, meaning that duplicate vectors will be removed before considering linear independence. Another view of this definition is to consider the minor $M$ of the adjacency matrix $A$ found by taking rows corresponding to $K$ and columns corresponding to $V(G) - K$. A new matrix $M'$ can be found by removing duplicate columns. If the columns of $M'$ are linearly independent, then $K$ is neighborhood independent. We consider a quick example and nonexample:

\begin{example}
    \label{cubocta neighborhood indep}
    Consider the vertex cover $K \subset V(G)$, where $G$ is the cuboctahedral graph in Figure \ref{fig:Neighborhood independence examples}, and $K$ is the set of vertices inducing the subgraph in orange. If we view $\rr^K$ as $\rr^8$, where the image under $w_1$ corresponds to the first coordinate, $w_2$ the second, and so on, we then have:
    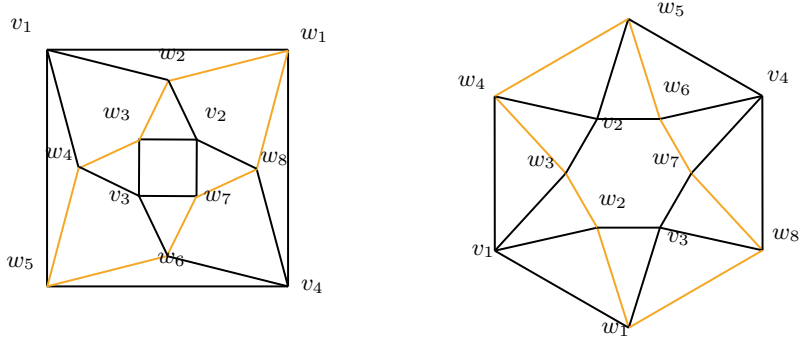
\begin{figure}
        \centering
        \tikzset{every picture/.style={line width=0.75pt}} 

\begin{tikzpicture}[x=0.75pt,y=0.75pt,yscale=-1,xscale=1]

\draw    (49.88,26.18) -- (49.88,145.65) ;
\draw    (171.43,26.18) -- (171.43,145.65) ;
\draw    (49.88,26.18) -- (171.43,26.18) ;
\draw    (49.88,145.65) -- (171.43,145.65) ;
\draw   (96.36,71.43) -- (125.4,71.57) -- (125.26,100.11) -- (96.22,99.97) -- cycle ;
\draw    (49.88,26.18) -- (65.76,85.59) ;
\draw    (155.55,86.24) -- (171.43,145.65) ;
\draw [color={rgb, 255:red, 245; green, 166; blue, 35 }  ,draw opacity=1 ]   (171.43,26.68) -- (155.55,86.74) ;
\draw [color={rgb, 255:red, 245; green, 166; blue, 35 }  ,draw opacity=1 ]   (65.76,86.09) -- (49.88,146.15) ;
\draw    (49.88,26.18) -- (110.99,41.49) ;
\draw    (110.33,130.34) -- (171.43,145.65) ;
\draw [color={rgb, 255:red, 245; green, 166; blue, 35 }  ,draw opacity=1 ]   (49.88,145.65) -- (110.33,130.34) ;
\draw [color={rgb, 255:red, 245; green, 166; blue, 35 }  ,draw opacity=1 ]   (110.99,41.99) -- (171.43,26.68) ;
\draw [color={rgb, 255:red, 245; green, 166; blue, 35 }  ,draw opacity=1 ]   (96.36,71.93) -- (65.76,86.09) ;
\draw [color={rgb, 255:red, 245; green, 166; blue, 35 }  ,draw opacity=1 ]   (155.86,86.44) -- (125.26,100.61) ;
\draw [color={rgb, 255:red, 245; green, 166; blue, 35 }  ,draw opacity=1 ]   (125.26,100.11) -- (110.33,130.34) ;
\draw [color={rgb, 255:red, 245; green, 166; blue, 35 }  ,draw opacity=1 ]   (110.99,41.99) -- (96.36,71.93) ;
\draw    (155.55,86.24) -- (125.4,71.57) ;
\draw    (96.22,99.97) -- (66.06,85.3) ;
\draw    (125.4,71.57) -- (110.99,41.49) ;
\draw    (110.63,130.05) -- (96.22,99.97) ;
\draw [color={rgb, 255:red, 245; green, 166; blue, 35 }  ,draw opacity=1 ]   (275.6,49.69) -- (311.5,88.57) ;
\draw [color={rgb, 255:red, 245; green, 166; blue, 35 }  ,draw opacity=1 ]   (374.82,88.57) -- (410.73,127.45) ;
\draw [color={rgb, 255:red, 245; green, 166; blue, 35 }  ,draw opacity=1 ]   (343.05,10.62) -- (358.99,61.15) ;
\draw [color={rgb, 255:red, 245; green, 166; blue, 35 }  ,draw opacity=1 ]   (327.33,115.99) -- (343.27,166.52) ;
\draw    (358.99,115.99) -- (343.27,166.52) ;
\draw    (343.05,10.62) -- (327.33,61.15) ;
\draw    (275.6,49.69) -- (327.33,61.15) ;
\draw    (358.99,61.15) -- (410.62,49.5) ;
\draw    (275.7,127.64) -- (327.33,115.99) ;
\draw    (275.7,127.64) -- (311.5,88.57) ;
\draw    (374.82,88.57) -- (410.62,49.5) ;
\draw    (358.99,115.99) -- (410.73,127.45) ;
\draw    (410.62,49.5) -- (410.73,127.45) ;
\draw    (343.05,10.62) -- (410.62,49.5) ;
\draw [color={rgb, 255:red, 245; green, 166; blue, 35 }  ,draw opacity=1 ]   (343.05,10.62) -- (275.6,49.69) ;
\draw [color={rgb, 255:red, 245; green, 166; blue, 35 }  ,draw opacity=1 ]   (410.73,127.45) -- (343.27,166.52) ;
\draw    (275.7,127.64) -- (343.27,166.52) ;
\draw    (275.6,49.69) -- (275.7,127.64) ;
\draw [color={rgb, 255:red, 245; green, 166; blue, 35 }  ,draw opacity=1 ]   (311.5,88.57) -- (327.33,115.99) ;
\draw [color={rgb, 255:red, 245; green, 166; blue, 35 }  ,draw opacity=1 ]   (358.99,61.15) -- (374.82,88.57) ;
\draw    (374.82,88.57) -- (358.99,115.99) ;
\draw    (327.33,61.15) -- (311.5,88.57) ;
\draw    (358.99,61.15) -- (327.33,61.15) ;
\draw    (358.99,115.99) -- (327.33,115.99) ;

\draw (30,9) node [anchor=north west][inner sep=0.75pt]  [font=\small] [align=left] {$\displaystyle v_{1}$};
\draw (128,54) node [anchor=north west][inner sep=0.75pt]  [font=\small] [align=left] {$\displaystyle v_{2}$};
\draw (79.5,96) node [anchor=north west][inner sep=0.75pt]  [font=\small] [align=left] {$\displaystyle v_{3}$};
\draw (176.5,141) node [anchor=north west][inner sep=0.75pt]  [font=\small] [align=left] {$\displaystyle v_{4}$};
\draw (176,14.32) node [anchor=north west][inner sep=0.75pt]  [font=\small] [align=left] {$\displaystyle w_{1}$};
\draw (104.5,24) node [anchor=north west][inner sep=0.75pt]  [font=\small] [align=left] {$\displaystyle w_{2}$};
\draw (47.5,74) node [anchor=north west][inner sep=0.75pt]  [font=\small] [align=left] {$\displaystyle w_{4}$};
\draw (28,129) node [anchor=north west][inner sep=0.75pt]  [font=\small] [align=left] {$\displaystyle w_{5}$};
\draw (104,127.32) node [anchor=north west][inner sep=0.75pt]  [font=\small] [align=left] {$\displaystyle w_{6}$};
\draw (127.5,95.82) node [anchor=north west][inner sep=0.75pt]  [font=\small] [align=left] {$\displaystyle w_{7}$};
\draw (156,76.32) node [anchor=north west][inner sep=0.75pt]  [font=\small] [align=left] {$\displaystyle w_{8}$};
\draw (76.5,53.82) node [anchor=north west][inner sep=0.75pt]  [font=\small] [align=left] {$\displaystyle w_{3}$};
\draw (262.59,122.9) node [anchor=north west][inner sep=0.75pt]  [font=\small]  {$v_{1}$};
\draw (328.09,58.9) node [anchor=north west][inner sep=0.75pt]  [font=\small]  {$v_{2}$};
\draw (361.09,116.58) node [anchor=north west][inner sep=0.75pt]  [font=\small]  {$v_{3}$};
\draw (411.59,35.08) node [anchor=north west][inner sep=0.75pt]  [font=\small]  {$v_{4}$};
\draw (328.09,162.58) node [anchor=north west][inner sep=0.75pt]  [font=\small]  {$w_{1}$};
\draw (326.59,98.04) node [anchor=north west][inner sep=0.75pt]  [font=\small]  {$w_{2}$};
\draw (290.59,78.04) node [anchor=north west][inner sep=0.75pt]  [font=\small]  {$w_{3}$};
\draw (255.59,38.04) node [anchor=north west][inner sep=0.75pt]  [font=\small]  {$w_{4}$};
\draw (356.09,1.04) node [anchor=north west][inner sep=0.75pt]  [font=\small]  {$w_{5}$};
\draw (359.09,40.9) node [anchor=north west][inner sep=0.75pt]  [font=\small]  {$w_{6}$};
\draw (353.59,77.4) node [anchor=north west][inner sep=0.75pt]  [font=\small]  {$w_{7}$};
\draw (414.09,114.9) node [anchor=north west][inner sep=0.75pt]  [font=\small]  {$w_{8}$};

\end{tikzpicture}
        \caption{Example subsets $K \subseteq V(G)$ for the cuboctahedron and antiprism graph $Ap_6$, respectively.}
        \label{fig:Neighborhood independence examples}
    \end{figure}
    $$
    \rho_{v_1} = \begin{bmatrix}
        1 \\ 1 \\ 0 \\ 1 \\ 1 \\ 0 \\ 0 \\ 0
    \end{bmatrix} \qquad \rho_{v_2} = \begin{bmatrix}
        0 \\ 1 \\ 1 \\ 0 \\ 0 \\ 0 \\ 1 \\ 1
    \end{bmatrix} \qquad \rho_{v_3} = \begin{bmatrix}
        0 \\ 0 \\ 1 \\ 1 \\ 0 \\ 1 \\ 1 \\ 0
    \end{bmatrix} \qquad \rho_{v_4} = \begin{bmatrix}
        1 \\ 0 \\ 0 \\ 0 \\ 1 \\ 1 \\ 0 \\ 1
    \end{bmatrix}.
    $$
    With some quick computation, it can be easily verified that these 4 vectors are linearly independent, and thus $K$ is neighborhood independent. 
\end{example}

\begin{example}
    \label{Ap_6 neighborhood dependence}
    Consider the vertex cover $K \subset V(G)$, where $G$ is the antiprism on 12 vertices $Ap_6$, as shown in Figure \ref{fig:Neighborhood independence examples} (a more formal definition of the Antiprism Graph is to come). As before, we construct the following vectors in $\rr^K$:
    $$\rho_{v_1} = \begin{bmatrix}
        1 \\ 1 \\ 1 \\ 1 \\ 0 \\ 0 \\ 0 \\ 0
    \end{bmatrix} \qquad \rho_{v_2} = \begin{bmatrix}
        0 \\ 0 \\ 1 \\ 1 \\ 1 \\ 1 \\ 0 \\ 0
    \end{bmatrix} \qquad \rho_{v_3} = \begin{bmatrix}
        1 \\ 1 \\ 0 \\ 0 \\ 0 \\ 0 \\ 1 \\ 1
    \end{bmatrix} \qquad \rho_{v_4} = \begin{bmatrix}
        0 \\ 0 \\ 0 \\ 0 \\ 1 \\ 1 \\ 1 \\ 1
    \end{bmatrix}.$$
    From here we see that $\rho_{v_1} - \rho_{v_2} - \rho_{v_3} + \rho_{v_4} = 0$, and thus $K$ is not neighborhood independent. 
\end{example}

We can now state our full conditions for establishing an upper bound on $B_3(G)$.

\begin{theorem}
    \label{upper bound theorem}
    For $G$ a $k$-regular graph, let $K \subseteq V(G)$ be a vertex cover of $G$ such that $K$ is neighborhood independent and the induced subgraph on $K$ is 2-edge-connected. Then $B_3(G) \leq \lvert K \rvert$.
\end{theorem}

\begin{proof}
    Since the induced subgraph $K$ is 2-edge-connected, we can choose a strong orientation $O$ on $K$ and construct the pot $P$ derived from $(K, O)$. Let $\lambda$ be the associated assembly design, so that $P=P_\lambda(G)$. We wish to show that $P$ is valid in Scenario 3. Let $K = \{v_1, \cdots, v_n\}$ and $\lambda(v_i) = t_i$ for $1 \leq i \leq n$. Since each $v_i \in K$ is a source vertex for a distinct bond-edge type in $\Sigma(P)$, let the corresponding bond-edge type for $v_i$ be $b_i$. Let all remaining tile types in $P$ (which may appear multiple times in the realization of $G$) be denoted $\ell_1, \cdots, \ell_d$. Note that $\#P = n + d$.
    
    Consider a graph $G' \in \mathcal{O}(P)$ such that $\#V(G') \leq \#V(G)$. Because any labeling of $G'$ must use at least one bond-edge type $b_i \in \Sigma(P)$, any assembling pot (that is a subset of $P$) for $G'$ will necessarily contain the tile $t_i$, which contains all of the unhatted $b_i$ cohesive-end types. Because $O$ is strongly oriented, for any other $v_j \in K$ there exists a directed path from $v_j$ to $v_i$. By Corollary \ref{directed path cor}, this implies that $t_j$ is in any assembling pot for $G'$. Thus, we have that tiles $t_1, \cdots, t_n$ will be contained in any subset of $P$ realizing $G'$.

    We now wish to show that exactly one of each $t_i$ can be used in a labeling of $G'$. For every tile type $t \in P$, we assign a number $g(t)$, which we define as the number of unhatted cohesive-end types in $t$ minus the number of hatted cohesive-end types in $t$. Note that $g(t)$ disregards the specific bond-edge types; it only considers whether each cohesive-end type is hatted or unhatted. We will also use the notation $g(v)$ for a vertex $v \in G$, which is understood to be the value $g(t)$ where $\lambda(v) = t$ in the assembly design of $G$ under consideration.
    Since $G'$ is assumed to be fully labeled, every half-edge labeled with a hatted cohesive-end type is bonded to a half-edge labeled with an unhatted cohesive-end type. Thus, it follows that: $\displaystyle\sum_{v \in V(G')} g(v) = 0$. For each tile type $t_1, \cdots, t_n$, choose a vertex labeled with that tile type in the realization of $G'$, and correspondingly label these vertices $v_1', \cdots, v_n'$. Let $K'$ be this set of vertices. 
    
    Note that the vertices of $K$ and $K'$ are labeled using the exact same tile types, and so we have that $\sum_{v \in K} g(v) = \sum_{v \in K'} g(v)$. Since $\sum_{v \in V(G)} g(v) =0 = \sum_{v \in V(G')} g(v)$ to maintain the balance of hatted and unhatted edges in the whole graph, we can subtract the first equation from the second to see that \begin{equation}
    \label{eq1}
    \sum_{v \in V(G') - K'} g(v) = \sum_{v \in V(G) - K} g(v).\end{equation}
    We note now that for every vertex $v \not\in K$, $\lambda(v) = \ell_i$, where $\ell_i$ consists of only hatted cohesive-end types. Since $G$ is $k$-regular, we have $g(v) = -k$ and $\displaystyle\sum_{v \in V(G) - K} g(v) = -k\lvert V(G) - K \rvert$. Combining this with Equation \ref{eq1} gives the following: \begin{equation}
    \label{eq2}
    \sum_{v' \in V(G') - K'}g(v') = -k\lvert V(G) - K \rvert .
    \end{equation} 
    Since each of the $t_i$ tile types contain at least one unhatted cohesive end-type, we know that $g(t_i) > -k$. The assumption $\#V(G) \geq \#V(G')$ implies $\lvert V(G) - K \rvert \geq \lvert V(G') - K' \rvert$, so the summation in Equation \ref{eq2} has at most $\lvert V(G) - K \rvert$ terms. Furthermore, since all $t_i$ are such that $g(t_i) > -k$ and all $\ell_i$ are such that $g(\ell_i) = -k$, we find that $g(v') \geq -k$ for all terms in the summation of Equation $\ref{eq2}$. Thus, for the sum in Equation \ref{eq2} to be equal to $-k\lvert V(G) - K \rvert$, it must be the case that every tile in $V(G') - K'$ is of tile type $\ell_i$ for some $i$. It also follows that $\lvert V(G) - K \rvert = \lvert V(G') - K' \rvert$, and therefore $\# V(G) = \# V(G')$.

    Lastly, we need to show that the composition of $\ell_1, \cdots, \ell_m$ used by $G'$ is identical to that of $G$. Let $h_i$ denote the number of times $\ell_i$ is used in $G$, and $h_i'$ denote the number of times $\ell_i$ is used in $G'$. Consider now the vector space $\rr^K$. Recalling our definition of neighborhood independence, we are interested in the set of vectors $\{\rho_v\}_{v \in V(G) - K}$. Notably, two vertices in $V(G) - K$ will correspond to the same vector $\rho_v$ if and only if they are labeled using the same tile type $\ell_{i}$, as they would thus have the same neighbor set within $V(G)$. We will accordingly call these vectors $\rho_i$ corresponding to each $\ell_i$. We consider now the sum:
    $$S = (h_1 - h_1')\rho_1 + (h_2 - h_2')\rho_2 + \cdots + (h_d - h_d') \rho_d.$$
    We wish to show that $S=0$. To do this, we will show $S(v_j) = 0$ for any $v_j \in K$. 
    
    Recall that $\rho_i(v_j)$ will equal 1 if a vertex with tile type $\ell_i$ is adjacent to $v_j$ and will equal $0$ otherwise. This is equivalent to $\ell_i$ containing cohesive-end type $\hat b_j$. Because $\ell_i$ has been used $h_i$ times in $\rho_i$, we have that $h_i \cdot \rho_i(j)$ is the number of $\hat b_j$ labeled half-edges in $V(G) - K$ contributed by the tile type $\ell_i$. Thus, the sum $h_1\cdot \rho_1(v_j) + \cdots + h_d\cdot \rho_d(v_j)$ is the total number of $\hat b_j$ half-edges of vertices in $V(G) - K$. Similarly, $h_1'\cdot \rho_1(v_j) + \cdots + h_d' \cdot \rho_d(v_j)$ is the total number of $\hat b_j$ half-edges of vertices in $V(G') - K'$. However, both $V(G) - K$ and $V(G') - K'$ need exactly enough $\hat b_j$ half-edges to match remaining $b_j$ half-edges in $K$ and $K'$, which are equal amounts. This gives the following.
    $$\begin{aligned}
    0&= h_1\cdot \rho_1(v_j) + \cdots + h_d\cdot \rho_d(v_j) - \left(h_1'\cdot \rho_1(v_j) + \cdots + h_d' \cdot \rho_d(v_j)\right) \\
    &= (h_1 - h_1')\rho_1(v_j) + \cdots + (h_d - h_d')\rho_d(v_j) = S(v_j).
    \end{aligned}$$
    Because this holds for every $v_j$, we have that $S = 0$. By the linear independence of $\rho_1, \cdots, \rho_d$, we have that $h_i = h_i'$ for all $i$. Thus, the realizations of $G$ and $G'$ use the same composition of tiles. Because all tile types have a unique source vertex, there is only one way to arrange the edges, and thus no nonisomorphic graphs can be constructed, which satisfies Scenario 3. 
\end{proof}

We observe that $k$-regularity is used very little in this proof, but the counting argument in the two paragraphs following Equation \ref{eq2} hinges on converting information about edge counts to information about vertex counts, which we cannot do if $G$ is of non-homogenous degree. We note that Theorem \ref{upper bound theorem} may be  understood through the following observation:

\begin{remark}
Suppose that $G$ a $k$-regular graph. If we are given a neighborhood independent vertex cover $K\subseteq V(G)$ such that the induced subgraph on $K$ is 2-edge-connected, then we may construct a valid pot for $G$ using $|K|$ bond-edge types. In particular, this pot is the one derived from $(K, O)$, where $O$ is chosen to be any orientation on the induced subgraph of $K$.
\end{remark}

\section{Specific Families of Graphs}
\label{section4}

We now move on to bounding and determining the values of $B_3$ and $T_3$ for two $k$-regular graph families. The lower and upper bounds discussed in Sections \ref{sec: lower} and \ref{sec: upper} will be applied for all of the graph families in this section, thus demonstrating the utility of these bounds in analyzing a select variety of graphs.

\subsection{Rook's Graphs}

Rook's graphs are a highly symmetric graph family that represent the set of all legal moves of the rook chess piece on a chessboard. Formally, they are defined as the following:

\begin{definition}
The \textbf{rook's graph} $R_{m, n}$ has vertices corresponding to the squares on an $m\times n$ grid. Two vertices are adjacent if and only if they lie on the same row or column of the grid. 
\end{definition}

\begin{figure}[H]
    \centering
    \includegraphics[scale=0.625]{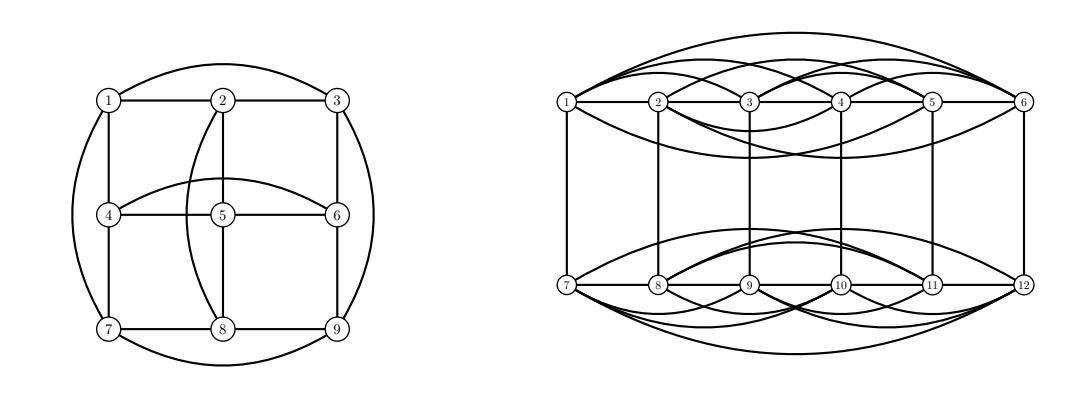}
    \caption{The rook's graph $R_{3, 3}$ (left) and the rook's graph $R_{2, 6}$ (right) \cite{rooksgraphassembly}.}
\end{figure}

Rook's graphs have been previously studied in Scenarios 1 and 2 in the flexible-tile model of DNA self-assembly (see \cite{rooksgraphassembly}). In this previous work, determining the values of $B_3$ and $T_3$ for rook's graphs in Scenario 3 was left as an open question. Using our upper and lower bounds, we may exactly determine $B_3$ and $T_3$ for the rook's graph family.

First, we use Theorem \ref{b3lowerbound} and Proposition \ref{t3lowerbound} to determine $T_3$ and a lower bound for $B_3$ for the rook's graphs:

\begin{proposition} \label{rooklowerbound}
    If $n, m \geq 4$, then $T_3(R_{m, n}) = mn$.
\end{proposition}

\begin{proof}
    We wish to show that $R_{m, n}$ is unswappable. Let $\{v_1, v_2\}$ and $\{v_3, v_4\}$ be two edges of $R_{m,n}$. If $v_1$ and $v_4$ (or $v_2$ and $v_3$) do not differ in both coordinates, then a multi-edge would be formed by swapping the edges; thus, the only way these edges could be swappable is if $v_1$ and $v_4$ differ in both coordinates. Thus, suppose that $v_1$ and $v_4$ differ in both coordinates. We wish to show that, even in this case, the edges are not swappable. 
    
    We proceed by counting copies of $K_4$ in our graph before and after swapping. First, consider $R_{m,n}$ with no swaps made. Note that $v_1$ and $v_2$ will differ in exactly one coordinate. Without loss of generality, suppose $v_1$ and $v_2$ have the same first coordinate. Since $m, n > 3$, there exist at least two other vertices corresponding to coordinate pairs that also have the same first coordinate as $v_1$ and $v_2$. Because of this, we know that $\{v_1, v_2\}$ is part of at least one copy of $K_4$. Similarly, $\{v_3, v_4\}$ will be part of at least one copy of $K_4$. 
    
    Now, consider the graph created by swapping the edges $\{v_1,v_2\}$ and $\{v_3,v_4\}$, such that the newly formed edges are $\{v_1, v_4\}$ and $\{v_3, v_2\}$. We will show that $\{v_1, v_4\}$ is not part of any new copies of $K_4$. As noted earlier, $v_1$ and $v_4$ differ in both coordinates. Assume, without loss of generality, these coordinates are $v_1 = (a, b)$ and $v_2 = (c, d)$. It follows that the only vertices adjacent to both are $(a, d)$ and $(c, b)$. However, these two vertices are not connected to each other. Thus, $\{v_1, v_4\}$ cannot create any new copies of $K_4$. A similar argument can be applied to $\{v_3, v_2\}$. Because copies of $K_4$ are removed, and no more are added, the graph formed as a result of the edge swap is not isomorphic to $R_{m, n}$. Thus, our graph is unswappable. 

    Finally, for every pair of vertices $v_1, v_2 \in R_{m, n}$, their neighbor sets must be different, as they can agree in at most one coordinate. It follows from Proposition \ref{t3lowerbound} that $T_3(R_{m, n}) = \#V(R_{m,n}) = mn$.
\end{proof}

Now, we use Theorem \ref{upper bound theorem} to upper bound $B_3$ for rook's graphs in order to exactly determine $B_3$.

\begin{proposition}  \label{rooksBbound}
    If $m, n \geq 3$ and $R_{m, n}$ is unswappable, then $B_3(R_{m, n}) = mn - n$.
\end{proposition}
\begin{proof}
    First, we note that the largest independent set of $R_{m, n}$ will be of size $n$, found by taking the collection of vertices at coordinates $(1, 1), (2, 2), \cdots, (n, n)$. Any larger set would include two vertices with equal first or second coordinates, which would be adjacent to one another. It follows that the smallest vertex cover is of size $mn - n$. By Theorem \ref{b3lowerbound},  $B_3(R_{m,n}) \geq mn-n$.

    Since we have already demonstrated that $B_3(R_{m, n}) \geq mn -n$, we need only show a matching upper bound. Consider the subset $K$ defined as follows: $K = V(G) - \{(i, i) | 1 \leq i \leq n\}$. First, this set will form a vertex cover, as for any edge $\{v_1, v_2\} \in R_{m, n}$, the vertices must differ in only one coordinate, and so they cannot both be of the form $(i, i)$. We apply Theorem \ref{upper bound theorem} using this vertex cover.
    
    We claim that the subgraph induced by $K$ is 2-edge-connected. Let $G$ denote the subgraph induced by $K$. First, we need to demonstrate that $G$ is connected. Consider two vertices $v, w \in G$ corresponding to coordinates $(a, b)$ and $(c, d)$, respectively. First, if either $a = c$ or $b = d$, then $v$ and $w$ are adjacent in $G$. Now, suppose $a \neq c$ and $b \neq d$. If $b \neq c$, then we can construct a path from $v$ to the vertex at coordinates $(c,b)$ to the vertex $w$. A similar path can be constructed if $a \neq d$. For the final case, suppose that $a = d$ and $b = c$. To restate, we wish to find a path $(a, b) \to (b, a)$. Because $m, n \geq 3$, we can choose some $1 \leq f \leq n$ such that $f \notin \{a, b\}$. Then $(a, b) \to (f, b) \to (f, a) \to (b, a)$ is a path in $G$.

    To show $2$-connectedness, consider an edge $\{v_1, v_2\}$ in $G$. We claim $K - \{v_1, v_2\}$ is connected. Because $K$ is connected, it is sufficient to show that there exists a path $v_1$ to $v_2$ within $K$. Suppose that $v_1 = (a, b)$ and $v_2 = (a, d)$ (the reasoning is analogous if the vertices differ in the first coordinate). Because these vertices are within $K$, we have that $a \neq b$ and $a \neq d$. We can thus observe the following path within $K$:
    $$(a, b) \to (d, b) \to (d, a) \to (b, a) \to (b, d) \to (a, d).$$
    Therefore, $G$ is 2-edge-connected. It only remains to be shown that $K$ is neighborhood independent. Recalling the definition of $\rho_v$ from Definition \ref{neighborhood indepedence}, we seek to show linear independence of the set $\{\rho_{(i, i)}\}_{1 \leq i \leq n}$. Suppose there exist $h_i \in \rr$ such that
    $$0 = \sum_{1 \leq i \leq n} h_i \rho_{i, i} = S \in \rr^K.$$
    We proceed by considering three arbitrary $1 \leq i_{1}, i_{2}, i_{3}\leq n$. Note that $0 = S((i_1, i_2)) = h_{i_1} + h_{i_2}$, as only $1 = \rho_{(i_1, i_1)}((i_1, i_2)) = \rho_{(i_1, i_1)}((i_1, i_2))$ will be nonzero. Similarly, $0 = h_{i_2} + h_{i_3}$ and $0 = h_{i_3} + h_{i_1}$. We now see that $h_{i_1} = -h_{i_2} = h_{i_3} = -h_{i_1}$, implying $h_{i_1}=0$. This process can be repeated to demonstrate $h_i = 0$ for each $i$, proving linear independence. 
    By Theorem \ref{upper bound theorem}, $B_3(R_{m,n}) \leq \#V(K) = mn-n$. Thus, $B_3(R_{m,n}) = mn-n.$
\end{proof}

Note that since $m, n\geq 4$ implies unswappability, Proposition \ref{rooksBbound} determines $B_3(R_{m, n}) = mn-n$ for all $m, n \geq 4$. With these two results, we have exactly determined the values of $B_3(R_{m, n})$ and $T_3(R_{m, n})$ for large enough $m,n$. This resolves the open problem posed in \cite{rooksgraphassembly}.

\begin{example}
We now consider an example pot construction for the rook's graph $R_{3, 3}$, following the proof of Theorem \ref{upper bound theorem}. We start by marking our vertex cover $K$, which will be the complement of the set $(1, 1), (2, 2), (3, 3)$. As stated in the proof, the complement must be 2-connected, and so we choose some strong orientation. We choose $(1, 2) \to (1, 3) \to (2, 3) \to (2, 1) \to (3, 1) \to (3, 2) \to (1, 2)$. To each of the vertices in the complement, we assign a bond-edge type. We assign:
$$(1, 2) \leftrightarrow o, 
\qquad (1, 3) \leftrightarrow b,
\qquad (2, 3) \leftrightarrow g,
\qquad (2, 1) \leftrightarrow p,
\qquad (3, 1) \leftrightarrow n,
\qquad (3, 2) \leftrightarrow r.
$$
This process is depicted in Figure \ref{fig:RooksExample}, and the choice bond-edge variables roughly correspond to the first letter of the color representing them (e.g. $o$ for orange). To complete the pot, orient the remaining edges away from $K$. Finally, assign to each edge the bond-edge type of the vertex in $K$ it originates from. This process leads to the following pot:
$$
P = \{\{o^3,\hat r\}, 
\{b^3, \hat o\}, 
\{g^3, \hat b\}, 
\{p^3, \hat g\},
\{n^3, \hat p\},
\{r^3, \hat o\},
\{\hat o, \hat n, \hat b, \hat p\},
\{\hat o, \hat r, \hat g, \hat p\},
\{\hat n, \hat r, \hat g, \hat b\}
\}.
$$
Here, the first six tiles correspond to vertices in the vertex cover, and the rest correspond to its complement. 

Note that Proposition \ref{rooklowerbound}  does not apply directly to this graph due to the small dimensions. However, since the graph can be checked to be unswappable via Appendix \ref{appendixA}, we may use Proposition \ref{t3lowerbound} to show that this pot minimizes tile types. Proposition \ref{rooksBbound} also shows that this pot minimizes bond-edge types, showing that this pot is optimal for both tile types and bond-edge types.

\begin{figure}
    \centering
    \input{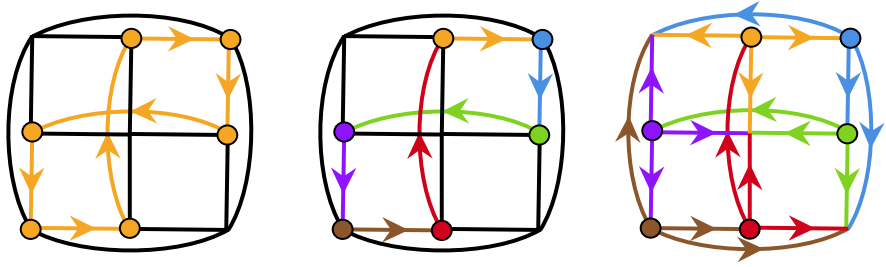}
    \caption{Steps for an optimal construction for $R_{3, 3}$.}
    \label{fig:RooksExample}
\end{figure}

\end{example}

\subsection{Kneser Graphs}

Next, consider the Kneser graph family. Intuitively, it seems that swapping any edge pairs would break the underlying combinatorial structure of these graphs. 

\begin{definition}
    The \textit{Kneser Graph} $Kn(n, i)$ for $n, i \in \zz$, $n > i > 0$, is defined in \cite{lovasz1978kneser} as follows: Fix the $n$ element set $[n] = \{1, \cdots, n\}$. The vertices of $Kn(n, i)$ are all subsets of $[n]$ of size $i$. An edge exists between two subsets of $[n]$, $A$ and $B$, if $A \cap B = \emptyset$. 
\end{definition}

One notable member of the Kneser graph family is $Kn(5, 2)$, which is isomorphic to the Petersen graph.

 \begin{figure}[H]
    \centering
    \import{./}{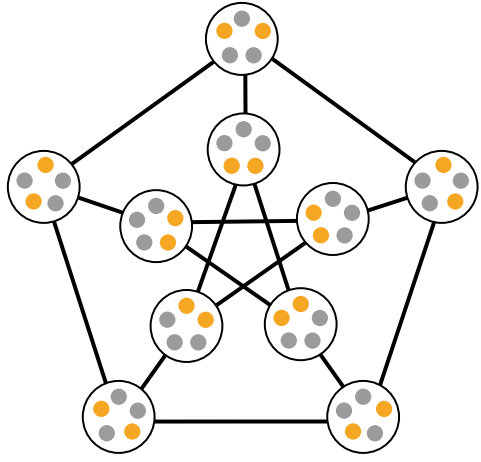}
    \caption{The graph $Kn(5, 2)$, also known as the Petersen graph. Each vertex is represented by a 2-element subset of $\{1, 2, 3, 4, 5\}$, and edges are between disjoint subsets.}
\end{figure}

The Petersen graph is a particularly clear example of unswappability. This is because it is known that the smallest cycle in the Petersen graph has length 5, while swapping any two disjoint edges creates a cycle of length 4, as shown previously in Figure \ref{petersen_unswap}. Using this as motivation, we may show that general Kneser graphs with some parameter constraints are unswappable.

\begin{lemma}
Let $Kn(n, k)$ denote the Kneser graph on $n, k$. If $k>1$ and $n\geq 3k$, then $Kn(n, k)$ is unswappable.
\end{lemma}
\begin{proof}
We claim that any swap of disjoint edges will change the number of triangles (i.e. 3-cycles) in $Kn(n,k)$, thus showing that any swap results in a nonisomorphic graph.

Suppose the edges $\{v_1, v_2\}$ and $\{v_3, v_4\}$ are swapped, such that the newly formed edges are $\{v_1, v_3\}$ and $\{v_2, v_4\}$. Let $V_1, V_2, V_3, V_4$ denote the subsets of $[n]$ that correspond to the vertices $v_1, v_2, v_3, v_4$ respectively. Note that $V_1, V_2$ are disjoint subsets of $[n]$ of size $k$, so $|V_1\cap V_2| = 2k$. This means that there are ${n-2k \choose k}$ other subsets of size $k$ that are also disjoint with $V_1$ and $V_2$ (choosing $k$ elements out of the $n-2k$ that are in neither $V_1$ or $V_2$). In particular, this shows that there are ${n-2k\choose k}$ triangles in the graph that use the edge $\{v_1, v_2\}$, and thus there are ${n-2k\choose k}$ triangles removed by removing this edge. Similarly, removing the edge $\{v_3, v_4\}$ also removes ${n-2k\choose k}$ triangles.

Note that if the newly formed edges $\{v_1, v_3\}$ or $\{v_2, v_4\}$ already existed in the original graph prior to swapping, then the swap would cause a multiple-edge, and therefore a nonisomorphic graph, to form. Thus, we assume that neither $\{v_1, v_3\}$ nor $\{v_2, v_4\}$ exists in $Kn(n, k)$. This means that $V_1$ and $V_3$ cannot be disjoint, and that $V_2$ and $V_4$ cannot be disjoint. In other words, $|V_1\cap V_3|\geq1$ and $|V_2\cap V_4| \geq 1$. 

After removing $\{v_1, v_2\}$ and $\{v_3, v_4\}$ from the graph, consider the number of new triangles created by adding the edges $\{v_1, v_3\}$ and $\{v_2, v_4\}$. For $\{v_1, v_3\}$, the number of triangles using this edge is equivalent to the number of subsets of size $k$ that are disjoint from both $v_1$ and $v_3$. The number of these subsets is ${n-2k+|V_1\cap V_3|\choose k}$. Since $V_2$ and $V_4$ might be included in these subsets, but the edges $\{v_1, v_2\}$ and $\{v_3, v_4\}$ were removed, the number of triangles using the edge $\{v_1, v_3\}$ is at least ${n-2k+|V_1\cap V_3|\choose k} - 2$. Furthermore, since we know that $|V_1\cap V_3| \geq 1$, the number of triangles using the edge $\{v_1, v_3\}$ is at least ${n-2k+1\choose k}-2$. The same holds for the number of triangles using the edge $\{v_2, v_4\}$.

Thus, the change in the number of triangles in the graph created by this swap is $2({n-2k+1\choose k}-2) - 2{n-2k\choose k}$. We claim that this is always greater than 0, thus showing that the number of triangles has changed in the graph. Showing this is equivalent to the following:  $${n-2k\choose k} < {n-2k+1\choose k}-2.$$

Using the relation ${n-2k+1\choose k} = {n-2k\choose k} + {n-2k\choose k-1}$, the above is equivalent to showing $$0 < {n-2k\choose k-1}- 2.$$ Note that since $k \geq 2$, we have that $k-1\geq 1$, and since $n\geq 3k$, $n-2k \geq k$. In particular, this means that ${n-2k\choose k-1} \geq {k\choose 1} = k$. This is greater than 2 for all choices except $Kn(6, 2)$ (i.e. $n=6$ and $k=2$). This particular choice has been confirmed to be unswappable via computation; thus, we have shown the desired result.
\end{proof}

\begin{proposition}
Let $Kn(n, k)$ denote the Kneser graph on $n, k$. If $k> 1$ and $n\geq 3k$, then $B_3(Kn(n, k)) \geq {n-1\choose k}$ and $T_3(Kn(n, k)) = {n\choose k}$.
\end{proposition}
\begin{proof}
By the Erdos-Ko-Rado Theorem \cite{erdoskorado1}, the maximal independent set of the Kneser graph has size ${n-1\choose k-1}$. Thus, the minimum vertex cover has size ${n\choose k} - {n-1\choose k-1} = {n-1\choose k}$, and $B_3(Kn(n, k)) \geq {n-1 \choose k}$ by Theorem \ref{b3lowerbound}. Since every vertex has a unique neighbor set (no two subsets are disjoint with exactly the same set of other subsets), we have $T_3 = \#V(Kn(n,k)) = {n \choose k}$ by Theorem \ref{t3lowerbound}.
\end{proof}

Thus, we have exactly determined the value of $T_3$ for certain Kneser graphs. In particular, we have found that every single vertex requires a unique tile type---a very strong result. Since the result for $B_3(Kn(n ,k))$ is only a lower bound, we may attempt to use Theorem \ref{upper bound theorem} to obtain a corresponding upper bound. In order to use this theorem, we need 2-edge connectivity. Kneser graphs were shown to be Hamiltonian in \cite{merino2023kneser}, which implies the 2-edge connectivity of the Kneser graphs. However, a simple way to apply the upper bound is only shown for a small subset of Kneser graphs in the following proposition.

\begin{proposition}
\label{oops}
For the Kneser graph $Kn(n, k)$ such that $n\geq 6$ and $k=2$, we have that $B_3(Kn(n, k)) = {n-1\choose k} = {n-1\choose 2}$.
\end{proposition}
\begin{proof}
Using the Erdos-Ko-Rado theorem, we have that a minimum vertex cover of $Kn(n, k)$ is all the vertices that are subsets of $\{1, 2, ..., n-1\}$ (so that the complement of the vertex cover is all sets containing $n$). Note that the induced subgraph of this vertex cover is $Kn(n-1, k)$ since there is still an edge between vertices in the cover if and only if they are disjoint; by \cite{merino2023kneser}, we have that this induced subgraph is 2-edge-connected. It only remains to be shown that $K$ is neighborhood independent. For this we wish to show that the set $\{\rho_{S} : S\subseteq[n]\cs |S| = k\cs n\in S\}$ is linearly independent, using $\rho(S)$ as defined in Definition \ref{neighborhood indepedence}.

To show this, consider the minimum vertex cover to be all subsets of $[n]$ of size 2 that do not include the element 1. Then, the complement of the vertex cover is all subsets that do include the element 1, and let this set be $S$. Consider $\{\rho_S\}$ as defined above, and suppose $\displaystyle\sum_{v\in S} h_v\rho_v = 0$.

Since $\displaystyle\sum_{v\in S} h_v\rho_v$ equals the zero vector, this sum will also equal the zero vector if we take a subset of the indices (corresponding to a subset of the vertex cover), then restrict each vector to these indices. In particular, let us restrict the indices to only the sets in the vertex cover that contain the element $n$. Let $\rho_v'$ denote the vector $\rho_v$, but restricted only to the indices that contain $n$. Then, the element $\{1, n\}\in S$ will give the zero vector $\rho_{\{1, n\}}' = 0$. Furthermore, all elements $\{1, i\}\in S$ for $i\neq n$ will be the all-ones vector, except for the index corresponding to $\{i, n\}$ in the vertex cover, which will contain 0. In particular, this means that if we order the restricted indices by $\{2, n\}\cs \{3, n\}\cs \{4, n\}\cs ...\cs \{n-1, n\}$, we have that $\rho_{\{1, 2\}}' = (0, 1, 1, ..., 1)$, $\rho_{\{1, 3\}}' = (1, 0, 1, 1, ..., 1)$, etc. until $\rho_{\{1, n-1\}}' = (1, 1, ..., 1, 0)$, since all indices are 1 for $\rho_{\{1, i\}}$ except the $i$-th index. As previously mentioned, we also have $\rho_{\{1, n\}}' = (0, 0, ..., 0)$. 

It can be easily verified that the set of vectors corresponding to $(0, 1, 1, ..., 1)$, $(1, 0, 1, ..., 1)$, which form the matrix $J_{n-1, n-1} - I_{n-1}$ (where $J_{n-1, n-1}$ is the all-ones matrix of dimension $n-1\times n-1$ and $I_{n-1}$ is the identity matrix of dimension $n-1\times n-1$), is linearly independent. In particular, this means that $h_{\{1, i\}} = 0$ for all $i\neq n$. Since $\rho_{\{1, n\}}'$ is the zero vector, $h_{\{1, n\}}$ is the only coefficient that can be nonzero. However, $\rho_{\{1, n\}}$ is not the zero vector when not restricted to our set of indices; thus, the linear combination $\displaystyle\sum_{v\in S} h_v\rho_v$ can only be equal to zero if $h_{\{1, n\}} = 0$ as well. Thus, we have shown $h_v = 0$ for all $v$, giving us the desired result of neighborhood independence. This shows that there exists a pot to match our lower bound for Kneser graphs on $k=2$.
\end{proof}

Thus, we have exactly determined the value $T_3(Kn(n, k))$ for large enough $n, k$, as well as determined $B_3(Kn(n, k))$ for large enough $n$ and $k=2$. The authors note that showing neighborhood independence for vertex covers becomes difficult for larger values of $k$, hence the limited subset of Kneser graphs in Proposition \ref{oops}. For an example of a pot for Kneser graphs, refer back to Figure \ref{petersen_pot}.

\section{Conclusion}
This paper focuses on providing general-use theorems for graphs with certain properties in Scenario 3, then applying these specifically to select $k$-regular graph families. In particular, we established a lower bound for unswappable graphs  (not necessarily regular) using a vertex cover model. We also provided a method for constructing pots for $k$-regular graph with certain constraints, which established an upper bound. Furthermore, by finding lower bounds and exact values for select graphs in Scenario 3, we have also implicitly provided upper bounds in Scenario 2 for graphs which values are not yet known.

Lastly, we pose a few directions for future research. The first is to try and develop a stronger understanding of what makes a graph ``swappable" or ``unswappable." In particular, is there a more efficient manner to check for unswappability? If so, can we formulate a hardness result? A second question is if we can generalize Theorem \ref{b3lowerbound} to non $k$-regular graphs. The regularity hypothesis is only used briefly, and thus it is natural to ask if it can be removed or weakened in some way. Even if a broad generalization cannot be found, there may be ways to adapt the proof to specific non-regular graphs or graph families.

\subsection{Acknowledgements}
We would like to thank the organizers, teaching assistants, ICERM, the NSF, and the undergraduates of the Summer@ICERM program for making this project possible.
\\\\
    This material is based upon work supported by the National Science Foundation under Grant No. DMS-1929284 while the authors were in residence at the Institute for Computational and Experimental Research in Mathematics in Providence, RI, during the Summer@ICERM program.

\bibliographystyle{ieeetr}
\bibliography{mybib}

\pagebreak

\appendix

\section{Implementation of Unswappability Checker}
\label{appendixA}
\begin{minted}[fontsize=\footnotesize]{python}
import numpy as np

def graphFormat(g):
    return [(int(pair.split('<->')[0]) - 1, int(pair.split('<->')[1]) - 1) for pair in g[1:-1].split(',')]

def numSpanningTrees(g):
    size = max(max(a) for a in g) + 1
    laplacian = [[0 for i in range(size)] for i in range(size)]
    for i, j in g:
        laplacian[i][j] -= 1
        laplacian[j][i] -= 1
        laplacian[i][i] += 1
        laplacian[j][j] += 1
    # print(np.array(laplacian, dtype=int))
    redLaplace = np.array(laplacian, dtype=int)[:size - 1, :size - 1]
    return int(np.rint(np.linalg.det(redLaplace)))

def nonMatch(g):
    matches = []
    for i in g:
        for j in g:
            if i[0] != j[0] and i[0] != j[1] and i[1] != j[0] and i[1] != j[1]:
                g_copy = g.copy()
                g_copy.remove(i)
                g_copy.remove(j)
                g_copy.append((i[0], j[1]))
                g_copy.append((i[1], j[0]))
                if numSpanningTrees(g) == numSpanningTrees(g_copy):
                    matches.append((i, j))
                g_copy = g.copy()
                g_copy.remove(i)
                g_copy.remove(j)
                g_copy.append((i[0], j[0]))
                g_copy.append((i[1], j[1]))
                if numSpanningTrees(g) == numSpanningTrees(g_copy):
                    matches.append((i, (j[1], j[0])))  
    return matches

g = graphFormat(graphOfInterest)
print(g)
print(numSpanningTrees(g))
print(nonMatch(g))
\end{minted}

\end{document}